\documentclass{cmslatex}
\usepackage[paperwidth=7in, paperheight=10in, margin=.875in]{geometry}
\usepackage[backref,colorlinks,linkcolor=red,anchorcolor=green,citecolor=blue]{hyperref}
\usepackage{amsfonts,amssymb}
\usepackage{amsmath,mathrsfs}
\usepackage{graphicx} 
\usepackage{cite} 
\renewcommand{\theequation}{\arabic{section}.\arabic{equation}}

\numberwithin{equation}{section}

\newcommand{\tr}{\top}

\newcommand{\ee}{\mathbb E}

\newcommand{\pp}{\mathbb P}
\newcommand{\nn}{\mathbb N}
\newcommand{\rr}{\mathbb R}

\newcommand{\yy}{\mathbb Y}

\newcommand{\CC}{\mathcal C}
\newcommand{\DD}{\mathcal D}

\newcommand{\LL}{\mathcal L}

\newcommand{\RR}{\mathcal R}
\newcommand{\HH}{\mathcal H}

\newcommand{\FFF}{\mathscr F}

\newcommand{\<}{\langle}
\renewcommand{\>}{\rangle}
\allowdisplaybreaks \allowdisplaybreaks[4]

\newcommand{\dd}{\mathrm{d}}

\newcommand{\abs}[1]{\left\lvert #1 \right\rvert}
\newcommand{\norm}[1]{\left\lVert #1 \right\rVert}

   \allowdisplaybreaks
\begin{document}
 \title
{Weak Error Estimates of Ergodic Approximations for Monotone Jump-diffusion SODEs\thanks{Received date: October 17, 2025, and accepted date: April 24, 2026.}}

\author{Zhihui LIU\thanks{Department of Mathematics \& National Center for Applied Mathematics Shenzhen (NCAMS) \& Shenzhen International Center for Mathematics, Southern University of Science and Technology, Shenzhen 518055, P.R. China, liuzh3@sustech.edu.cn.}
  \and Xiaoming WU\thanks{Department of Mathematics, Southern University of Science and Technology, Shenzhen 518055, P.R. China, 12331004@mail.sustech.edu.cn.}}

\pagestyle{myheadings} 
\markboth{Weak Error Estimates of Ergodic Approximations for SODEs}{Zhihui LIU and Xiaoming WU } 

\maketitle

\begin{abstract}
We first derive the exponential ergodicity of the stochastic theta method (STM) with $\theta \in (1/2,1]$ for monotone jump-diffusion stochastic ordinary differential equations (SODEs) under a dissipative condition. 
Then we establish the weak error estimates for the backward Euler method (BEM), corresponding to the STM with $\theta=1$.
In particular, the time-independent estimate for the BEM in the jump-free case gives a first-order convergence rate between the exact and numerical invariant measures, answering a question left in {\it Z. Liu and Z. Liu, J. Sci. Comput. (2025) 103:87}. 
\end{abstract}



\section{Introduction}

\vspace{0.5em}

While SODEs driven by Gaussian noise are widely used in physics, finance, engineering, and other scientific fields, real-world systems are often exposed to event-driven uncertainties.
Such phenomena necessitate the modeling through jump components \cite{Rong05}.
For instance, stock prices can be subject to sudden, significant movements driven by unpredictable events such as market crashes, central bank announcements, and changes in credit ratings. 
Since exact solutions of nonlinear SODEs with jumps are rarely available, numerical methods have become a powerful tool for analyzing their behavior.

The present  paper concerns the weak error analysis of ergodic approximations for the following jump-diffusion SODEs:
\begin{align}\label{sde}
    \dd X(t) = b(X(t-)) \dd t + \sigma(X(t-)) \dd W(t)+\int_Z\gamma(X(t-),z) \widetilde{N}(\dd t, \dd z), 
\end{align}
where $X(t-):=\lim_{s \rightarrow t^{-}} X(s)$.
Here $b: \rr^d \rightarrow \rr^d$, $\sigma:\rr^d \rightarrow \rr^{d \times m}$, $\gamma: \rr^d\times Z \rightarrow \rr^d$ are muturally measurable functions, and $W$ and $\widetilde{N}$ are independent $\rr^m$-valued Wiener process and compensated Poisson random measure with the jump intensity $\nu \neq 0$ and $\nu(Z)<\infty$, respectively, on a complete filtered probability space $(\Omega, \FFF, \{\FFF(t)\}_{t\geq 0}, \pp)$. 
 
The long-time behavior of Markov processes generated by SDEs is a fundamental question in many scientific areas. 
As a particularly significant form of long-term behavior, ergodicity characterizes the case of the temporal average coinciding with the spatial average. This property finds broad applications in quantum mechanics, fluid dynamics, financial mathematics, and many other fields \cite{DZ96}. 
Much progress has been made in the design of numerical methods capable of inheriting the ergodicity of Eq. \eqref{sde} with non-Lipschitz coefficients in the jump-free case (i.e., $\gamma=0$), including the tamed schemes and STM studied in, 
e.g., \cite{LMW23,Liu25,LL25,LW24,MSH02} and references therein. 
In these developments, the numerical Lyapunov structure plays a key role.
Our first main aim in this work is to establish the numerical Lyapunov structure and exponential ergodicity of the STM \eqref{stm} with $\theta \in (1/2, 1]$ under a dissipative assumption (see Theorem \ref{tm-inv}). 

To quantify the dynamical impact of the BEM, corresponding to the STM \eqref{stm} with $\theta=1$, one needs to analyze its weak error.
In the past two decades, various numerical methods have been developed and analyzed for jump-diffusion SODEs, see, e.g., \cite{CG20, DFLM19, HK07, MZ11, PS21}, most of which focus primarily on strong approximations.
It should be noted that weak error estimates rely heavily on high-order moment estimates, which are challenging for the STM because the coefficients are superlinear. 
Our second aim is to present a finite-time weak error analysis between the BEM and the jump-diffusion SODEs \eqref{sde} (see Theorem \ref{tm-weak}).
It should be pointed out that the infinite-time weak error analysis poses significant challenges and remains open, mainly due to the lack of uniform moment bounds for the STM \eqref{stm} and the uniform regularity of the corresponding Kolmogorov equation over an infinite horizon.
Our final aim is to obtain the time-independent estimate for the BEM in the jump-free case, thereby obtaining a first-order convergence between the exact and numerical invariant measures (see Corollary \ref{cor-inv}).

The paper is organized as follows.
In Section \ref{sec2}, we derive the numerical exponential ergodicity of the STM \eqref{stm} via its Lyapunov structure, Feller property, and exponential dependence estimate. 
 Section \ref{sec3} investigates the moment estimates for the BEM \eqref{bem} and the regularity of the Kolmogorov equation associated with Eq. \eqref{sde}.
The weak error analysis for the BEM is established in Section \ref{sec4}.

\section{Exponential Ergodicity of STM } 
 \label{sec2}
 
\vspace{0.5em}

Let us begin with some frequently used notations.
We denote by $\abs{\cdot}$ and $\<\cdot, \cdot\>$ the Euclidean norm and inner product in $\rr$, $\rr^d$, or $\rr^m$ if there is no confusion, and by $\norm{\cdot}$ the Hilbert--Schmidt norm in $\rr^{d \times m}$. 
 In addition, we denote by $C_b(\rr^d)$ and $\mathrm{Lip}(\rr^d)$ the Banach spaces of all bounded, uniformly continuous mappings and Lipschitz continuous mappings $\varphi:\rr^d \rightarrow \rr$ endowed with the seminorms $|\varphi|_0=\sup_{x\in\rr^d}|\varphi(x)|$ and $|\varphi|_{\mathrm{Lip}}=\sup_{x\ne y}|\varphi(x)-\varphi(y)| |x-y|^{-1}$, respectively. 
 Moreover, for a $k \in \nn_+$, denote by $C_b^k(\rr^d)$ the subspace of $C_b(\rr^d)$ consisting of all functions with bounded partial derivatives $D^i\varphi(\cdot): \rr^d \to \LL((\rr^d)^{\otimes i}; \rr)$ for $1\le i\le k$, and with the norm $|\varphi|_k:=|\varphi|_0+\sum_{i=1}^k \sup_{x\in\rr^d} \|D^i\varphi(x)\|_{\LL((\rr^d)^{\otimes i}; \rr)}$.  

Let $\tau \in (0, 1)$ be a fixed step-size.  
Our main focus is on the following STM with $\theta \in [1/2, 1]$, studied in \cite{LL25} for the jump-free case:
\begin{align} \label{stm}
   Y_{j+1} &= Y_j  + (1-\theta)\tau b(Y_j )  
    + \theta \tau b(Y_{j+1})  + 
    \sigma(Y_j ) \delta_j W
     +
    \int_{t_j}^{t_{j+1}}\int_Z\gamma(Y_j,z) \widetilde{N}(\dd s, \dd z),
    \end{align} 
    where $t_j=j \tau$, and $\delta_j W:=W(t_{j+1})-W(t_j)$, $j \in \nn$.
    This section will investigate the numerical ergodicity of \eqref{stm} by exploring its Lyapunov structure, Feller property, and exponential dependence estimate. 

We need the following coupled monotone and coercive conditions to develop the Lyapunov structure of the STM \eqref{stm}. 

\vspace{0.5em}

   \begin{assumption} \label{A1}
There exist constants $L_i>0$, $i=1,2,3$, such that for any $x,y\in \rr^d$,  
  \begin{align} 
      2\<x-y,b(x)-b(y)\>+ \|\sigma(x)-\sigma(y)\|^2 +
 &  \int_Z|\gamma(x,z)-\gamma(y,z)|^2 \nu(\dd z)  
   \le L_1 |x-y|^2, \label{mon-} \\
      2\<x,b(x)\> + \norm{\sigma(x)}^2 + 
      &  \int_Z|\gamma(x,z)|^2 \nu(\dd z)
 \le  L_2 - L_3 \abs{x}^2.  \label{coe}
  \end{align}  
\end{assumption}

Under Assumption \ref{A1} and certain integrability conditions on the coefficients, \cite{GK80} showed that Eq. \eqref{sde} admits a unique solution whose trajectories are almost surely (a.s.) c\`adl\`ag (right-continuous with left-limits).
Similarly to the jump-free case studied in \cite{LL25},  the following lemma presents the well-posedness of the STM \eqref{stm}.

\vspace{0.5em}

\begin{lem}
Under the condition \eqref{mon-} and $L_1\theta \tau<2$, the STM  \eqref{stm} can be uniquely solved a.s. and $\{Y_n\}_{n \in \nn}$ is a homogeneous $\{\FFF_n:= \FFF(t_n)\}_{n \in\nn}$-Markov chain. 
\end{lem}

\subsection{Lyapunov structure.} 
We begin with the following technical inequality, which will be helpful to explore the Lyapunov structure of the STM \eqref{stm}.

\vspace{0.5em}

\begin{lem}\label{lm-in-stm}
  Let $\theta\in(1/2,1]$ and assume \eqref{coe} holds. 
  For any $\alpha_1 \ge 0$ and $\tau \in(0,1)$ with $2\alpha_1\le L_3$ and $2\alpha_1 \theta^2 \tau \le 2\theta -1$ and for all $x \in \rr^d$, it holds that 
  \begin{align}\label{in-stm}
        &  |x+(1-\theta) \tau b(x)|^2 + \tau \Big[\|\sigma(x)\|^2+ \int_Z|\gamma(x,z)|^2 \nu(\dd z) \Big]
        \le (1-\alpha_1\tau) |x-\theta \tau b(x)|^2 +  L_2\tau.
      \end{align}   
\end{lem}

\begin{proof}
Let $x \in \rr^d$. By separating $x$ and $b(x)$ and using \eqref{coe}, we have
    \begin{align*}
       &  |x+(1-\theta) \tau b(x)|^2 + \tau \Big[\|\sigma(x)\|^2+ \int_Z|\gamma(x,z)|^2 \nu(\dd z) \Big]  -  (1-\alpha_1\tau) |x-\theta \tau b(x)|^2 - L_2\tau \\
        = & \alpha_1\tau |x|^2+\tau
         \Big[2\<x,b(x)\>- L_2+ \|\sigma(x)\|^2+ \int_Z|\gamma(x,z)|^2 \nu(\dd z)\Big] \\
& \quad - 2\alpha_1\theta\tau^2\<x,b(x)\>+ (1-2\theta+\alpha_1\theta^2\tau)|\tau b(x)|^2 \\
& \le (2\alpha_1-L_3)\tau|x|^2 + 
         (1-2\theta+2\alpha_1\theta^2\tau)|\tau b(x)|^2,
    \end{align*} 
    which  is non-positive, as $2\alpha_1\le L_3$ and $2\alpha_1 \theta^2 \tau \le 2\theta -1$.
    Then we conclude \eqref{in-stm}.
\end{proof}

\vspace{0.5em}

For $\tau \in (0, 1)$ and $\theta\in [1/2,1]$, define a function $V_\theta:\rr^d \to [0,\infty)$ by 
\begin{align}\label{lya}
      V_\theta(x) := |x-\theta\tau b(x) |^2,
      \quad x \in \rr^d.
  \end{align} 
Similarly to the proof of \cite[Theorem 1]{LL25}, we have 
$V_\theta(x) \ge  (1+\theta L_3 \tau)|x|^2-L_2\theta \tau,$
      which shows that $\lim_{x \to \infty} V_\theta(x)=\infty$; therefore, $V_\theta$ is a Lyapunov function.
      
We have the following strong Lyapunov structure for the STM \eqref{stm} with $\theta\in (1/2,1]$, corresponding to $V_\theta$ defined by \eqref{lya}.
Throughout, we assume that $X_0$ is $\FFF_0$-measurable and we use $\ee_n[\cdot]:=\ee[\cdot|\FFF_n]$ to denote the conditional expectation with respect to $\FFF_n$.

\vspace{0.5em}

\begin{thm} \label{tm-lya}
Assume $\ee V_\theta(X_0)<\infty$ and let Assumption \ref{A1}  hold. 
  For any $\tau \in (0,1) $ and $\theta\in (1/2,1]$ satisfying $L_1\theta \tau<2$ and $n \in \nn$,   
  \begin{align}\label{lya+}
      \ee [V_\theta(Y_{n+1})\mid \FFF_{n}] 
      \le  [1-(2\theta-1)L_3\tau/2 ]V_\theta (Y_n) +  L_2 \tau.
  \end{align}  
\end{thm} 
\begin{proof}
Let $n \in \nn$.
To lighten the notation, we write $b_n:=b(Y_n)$, $\sigma_n:=\sigma(Y_n)$, $\gamma_n(z):=\gamma(Y_n,z)$, and $\delta_n \widetilde{N}(Y_n):=\int_{t_n}^{t_{n+1}}\int_Z \gamma_n(z) \widetilde{N}(\dd s, \dd z)$.
From \eqref{stm} we observe that
\begin{align} \label{eq-stm}
   |Y_{n+1}-\theta b_{n+1} \tau|^2 
  &=  |Y_n + (1-\theta) b_n \tau  + \sigma _n \delta_nW
  + \delta_n \widetilde{N}(Y_n)|^2 \nonumber\\ 
    &=|Y_n +(1-\theta )b_n \tau|^2
    + \tau \Big[\|\sigma_n\|^2 + \int_Z |\gamma_n(z)|^2\nu(\dd z) \Big] + H_n,
\end{align}
where 
\begin{align*}
H_n :&= 2 \<Y_n + (1-\theta)b_n \tau,\sigma_n
  \delta_n W+ \delta_n \widetilde{N}(Y_n)\> 
  +|\sigma _n \delta_nW + \delta_n \widetilde{N}(Y_n)|^2 \\
  & \quad - \tau \Big[\|\sigma_n\|^2 +\int_Z|\gamma_n(z)|^2\nu(\dd z)\Big].
\end{align*}

Because the Wiener and Poisson processes are mutually independent and the compensated Poisson random measure is a martingale, we have $\ee_n[H_n]=0$.
Then, taking the conditional expectation on both sides of the equality \eqref{eq-stm}, we conclude \eqref{lya+} by applying Lemma \ref{lm-in-stm} with $\alpha_1=(2\theta-1)L_3/2$ under the condition $L_3\theta^2\tau\le1$ (so that $2\alpha_1\le L_3$ and $2\alpha_1 \theta^2 \tau \le 2\theta -1$ hold). 

Note that the previous condition $L_3\theta^2\tau\le1$ can be relaxed because \eqref{coe} yields the same inequality with $L_3$ replaced by $L_3'$ for any positive $L_3'<L_3$ (so that one can always take $L_3=L_3'$ such that $L_3' \theta^2\tau\le1$).
     \end{proof}

\vspace{0.5em}

Similarly to \cite{LL25}, we have the following weak Lyapunov structure for the mid-point scheme corresponding to the STM \eqref{stm} with $\theta=1/2$.  

\vspace{0.5em}

\begin{cor}  \label{cor-lya}
Assume $\ee V_{1/2}(X_0)<\infty$ and let the condition \eqref{coe} hold. 
  For any $\tau \in (0,1) $ with $L_1 \tau < 4$ and $n \in \nn$, 
  \begin{align*}
      \ee [V_{1/2}(Y_{n+1})\mid \FFF_{n}] \le 
      V_{1/2} (Y_n) +  L_2 \tau-L_3|Y_n|^2.
  \end{align*}  
\end{cor}

\vspace{-1em}

 \subsection{Feller Property of STM.} 
Denote by $\{Y_k^x: k\in\nn\}$ and $\{Y_k^y: k\in\nn\}$ the solution of the STM \eqref{stm} starting from $x$ and $y$, respectively.

The following Feller property of the STM \eqref{stm}, in combination with the Lyapunov structure given in Theorem \ref{tm-lya} and Corollary \ref{cor-lya}, guarantees the existence of an invariant measure of \eqref{stm} for $\theta\in [1/2,1]$
(see \cite[Theorem 7.1 and Proposition 7.10]{DA06}).
  
\vspace{0.5em}

 \begin{thm} \label{tm-fel}
Assume Assumption \ref{A1} holds.  
Then for any $\tau \in (0,1)$ and $\theta\in [1/2,1]$ with $L_1\theta \tau \le 1/2$ and $n \in \nn$, 
  \begin{align}\label{fel}
      \ee|Y_n^x-Y_n^y|^2   
      &\le \frac{e^{2L_1t_n}}{1-L_1\theta\tau} |(x-y)- \theta \tau [b(x)-b(y)]|^2.
  \end{align} 
  Consequently, the corresponding discrete Markov semigroup of the STM \eqref{stm} is Feller, provided $b$ is continuous, and possesses an invariant measure.
\end{thm}

\vspace{0.5em}

\begin{proof}
As in the proof of Theorem \ref{tm-lya}, set $b_n^x:=b(Y_n^x), b_n^y:=b(Y_n^y),\sigma_n^x:=\sigma(Y_n^x),\sigma_n^y:=\sigma(Y_n^y),\gamma_n^x(z):=\gamma(Y_n^x,z),\gamma_n^y(z):=\gamma(Y_n^y,z)$  and $E_n:=Y_n^x-Y_n^y$.  Similarly, $E_n(b):=b_n^x-b_n^y$, $E_n(\sigma):=\sigma_n^x-\sigma_n^y$, $E_n(\gamma):=\gamma_n^x(z)-\gamma_n^y(z)$, and $\delta_n \widetilde{N}:=\int_{t_n}^{t_{n+1}}\int_Z E_n(\gamma) \widetilde{N}(\dd s, \dd z)$.

By \eqref{stm}, we have
\begin{align*}
 &\quad |E_{n+1}-\theta\tau E_{n+1}(b)|^2  \\  
  &=  |E_n + (1-\theta)\tau E_n(b) |^2 + |E_n(\sigma)\delta_n W|^2 
  + |\delta_n \widetilde{N}|^2 \\
  &\quad + 2 \<E_n + (1-\theta)\tau E_n(b) ,E_n(\sigma)
  \delta_n W+ \delta_n \widetilde{N}\>
  +2 \<E_n(\sigma) \delta_n W, \delta_n \widetilde{N}\> \\
 & =  |E_n - \theta\tau E_n(b)|^2 + 
  2 \<E_n,E_n(b)\> \tau +|\delta_n \widetilde{N}|^2
  + |E_n(\sigma)\delta_n W|^2 +(1-2\theta)\tau|E_n(b)|^2 \\
  &\quad   + 2 \<E_n + (1-\theta)\tau E_n(b),E_n(\sigma)
  \delta_n W+ \delta_n \widetilde{N}\>
  +2 \<E_n(\sigma) \delta_n W, \delta_n \widetilde{N}\>.
\end{align*}
Taking the conditional expectation on both sides of the above equality and using It\^o isometry, the condition \eqref{mon-}, the fact $\theta \ge 1/2$, and the assumption $2 L_1\theta \tau \le 1$, we have
 \begin{align*}
 & \ee_n |E_{n+1}-\theta\tau E_{n+1}(b)|^2  \\  
 &=|E_n - \theta\tau E_n(b)|^2 + 
  (1-2\theta)|E_n(b)\tau|^2
   + \tau \Big[2 \<E_n,E_n(b)\> + \|E_n(\sigma)\|^2 
 + \int_Z|E_n(\gamma)|^2\nu(\dd z) \Big] \\
 &\le |E_n - \theta\tau E_n(b)|^2+L_1\tau|E_n|^2 \\
& = (1+2L_1\tau)|E_n - \theta\tau E_n(b)|^2 
-L_1\tau [|E_n|^2 -4 \<E_n, \theta\tau E_n(b)\>
+ 2 |\theta\tau E_n(b)|^2] \\ 
& \le (1+2L_1\tau)|E_n - \theta\tau E_n(b)|^2- (1-2 L_1\theta \tau)L_1\tau|E_n|^2 \\
& \le (1+2L_1\tau)|E_n - \theta\tau E_n(b)|^2.
 \end{align*} 
Then we conclude \eqref{fel} from  the elementary inequality $(1+x)^n < e^{nx}$, $\forall~ n \in \nn_+,~x \ge 0$, and $|E_{n+1}-\theta\tau E_{n+1}(b)|^2 \ge (1-L_1 \theta \tau) |E_n|^2$ followed from \eqref{mon-}. 
  \end{proof}

 \vspace{0.5em}
 
 \subsection{Exponential Ergodicity of STM.}
  In this part, we focus on the exponential ergodicity of the STM \eqref{stm} under the following dissipative condition, which is stronger than Assumption \ref{A1}.

\vspace{0.5em}

  \begin{assumption} \label{A2}
There exists a positive constant  $L_1^*$ such that for all $x,y\in \rr^d$,  
  \begin{align*} 
 2\<x-y,b(x)-b(y)\>+ \|\sigma(x)-\sigma(y)\|^2 +
  \int_Z|\gamma(x,z)-\gamma(y,z)|^2 \nu(\dd z)\le -L_1^* |x-y|^2.
  \end{align*} 
\end{assumption}

\vspace{-0.5em}

In the continuous case, applying  It\^o formula implies that $\{e^{C_0 t}|X_t^x-X_t^y|^2: t \ge 0\}$ is a supermartingale for some positive constant $C_0$.
However, it remains unknown whether $\{e^{C_0 t_n}|Y_n^x-Y_n^y|^2: n \in \nn\}$ is a discrete supermartingale.
Fortunately, under  Assumption  \ref{A2}, we obtain the following uniform continuity dependence estimate for the STM \eqref{stm}.

\vspace{0.5em}

 \begin{thm} \label{tm-depen}
  Let Assumption \ref{A2} hold.
For any $\tau \in (0, 1)$ and $\theta\in (1/2,1]$ with $L_1^* \theta^2\tau\le 3-2\theta$ and $n \in \nn$,
  \begin{align}\label{depen+}
      \ee |Y_n^x-Y_n^y|^2 
      \le \frac{e^{- (\theta-1/2)L_1^* t_n}}{1+L_1^*\theta\tau} |(x-y)- \theta \tau [b(x)-b(y)]|^2.
  \end{align}  
\end{thm}

\begin{proof}
As $\theta \le 1$ and $L_1^*\tau \theta^2 \le 3-2\theta$, adopting the same notations and techniques as in the proof of Theorem \ref{tm-fel}, we have 
\begin{align*}
 &\quad \ee_n |E_{n+1}-\theta\tau E_{n+1}(b)|^2  \nonumber \\
  &\le |E_n -\theta\tau E_n(b)|^2 +(1-2\theta)|\tau E_n(b) |^2-L_1^*\tau|E_n|^2 \\
  & = [1-(\theta-1/2)L_1^*\tau]  |E_n -\theta \tau E_n(b) |^2 - (3/2-\theta)L_1^* \tau|E_n|^2  \\
         &\quad  -(2\theta-1)[L_1^*\tau \<E_n,\theta\tau E_n(b)\>+(1-L_1^* \theta^2 \tau/2)|\tau E_n(b)|^2] \\
& = [1-(\theta-1/2)L_1^*\tau]  |E_n -\theta\tau E_n(b) |^2 
+\frac{2\theta-1}{3-2\theta} [L_1^*\tau \theta^2-(3-2\theta)] |E_n|^2 \\
         &\quad  - (3/2-\theta)L_1^*\tau \Big|E_n+\frac{2\theta-1}{3-2\theta} \theta \tau E_n(b) \Big|^2  \\
         & \le [1-(\theta-1/2)L_1^*\tau]  |E_n -\theta\tau E_n(b) |^2.
\end{align*}
This shows \eqref{depen+} using a similar idea in the last part of the proof of Theorem \ref{tm-fel}. 
  \end{proof}

\vspace{0.5em}

\begin{thm}\label{tm-inv}
Let Assumption \ref{A2} hold and assume that $b$ is continuous.
For any $\tau \in (0, 1)$ and $\theta\in (1/2,1]$ with $L_1^* \theta^2\tau\le 3-2\theta$, the STM  \eqref{stm} is exponentially mixing with respect to a unique invariant measure, denoted by $\pi_\tau$.
\end{thm}

\vspace{0.5em}

\begin{proof}
By Theorem \ref{tm-fel}, the STM \eqref{stm} possesses an invariant measure $\pi_\tau$.
For any $\varphi \in {\rm Lip}(\rr^d)$, the continuity dependence estimate \eqref{depen+} implies 
\begin{align*}
|\ee\varphi(Y_n^x)-\pi_\tau(\varphi)|
=\Big| \int_{\rr^d} [\varphi(Y_n^x)-\varphi(Y_n^y)] \pi_\tau(\dd y)\Big|
\le \sqrt{\frac{e^{- (\theta-1/2)L_1^* t_n}}{1+L_1^*\theta\tau}} |\varphi|_{\mathrm{Lip}}c(x),
\end{align*}
for some nonegative measurable function $c: \rr \to \rr_+$. 
This shows the exponential mixing and thus the uniqueness of the invariant measure of \eqref{stm}. 
\end{proof}

\section{ Estimates of BEM and Kolmogorov Equation}
\label{sec3}

\vspace{0.5em}

In this section, we aim to give a high-order moment estimate of the BEM (i.e., the STM \eqref{bem} with $\theta=1$) 
\begin{align}\label{bem}
   Y_{j+1} = Y_j   
    + \ b(Y_{j+1}) \tau + 
    \sigma(Y_j ) \delta_j W +
    \int_{t_j}^{t_{j+1}}\int_Z\gamma(Y_j,z) \widetilde{N}(\dd s, \dd z),
    \quad j \in \nn,
    \end{align}  
and the regularity of the (backward) Kolmogorov equation associated to Eq. \eqref{sde} on any finite time interval $[0,T]$, where $T$ is a given positive number.
Additionally, we intend to derive uniform moment and regularity estimates in the jump-free case.

We need the following monotonicity and Lipschitz continuity conditions on the coefficients in Eq. \eqref{sde} for these two purposes.

\vspace{0.5em}
 
\begin{assumption} \label{A3}
There exist positive constants $L_4$ and $L_5$ such that for any $x, y \in \rr^d$ and $z\in Z$,
\begin{align}
\<x-y,b(x)-b(y)\>\le L_4|x-y|^2,  \label{coup}\\ 
\|\sigma(x)-\sigma(y)\|^2 + 
|\gamma(x,z)-\gamma(y,z)|^2 
& \le L_5|x-y|^2.   \label{sg-lip}
\end{align}
\end{assumption}

\begin{rem}\label{rk-coe}
(1) Under Assumption \ref{A3}, the solution $X$ of Eq. \eqref{sde} possesses bounded moments provided that $\ee|X_0|^{2p} < \infty$ for some $p \in \nn_+$ (see \cite{CG20}):
 \begin{align*}
 \ee|X(t)|^{2p}\le e^{Ct}(1+\ee|X_0|^{2p}), \quad t \in [ 0, T].
 \end{align*} 
 
(2) By \eqref{sg-lip}, there exists a positive constant $L_6$ such that for any $z\in Z$,
 \begin{align} \label{sg-grow}
   \|\sigma(x)\|^2 + |\gamma(x,z)|^2 
   &\le L_6(1+|x|^2), \quad x \in \rr^d. 
 \end{align}
If $b$, $\sigma, \gamma(\cdot, z) \in \CC^1(\rr^d)$ for any $z \in Z$, then for any $x, y \in \rr^d$,  
 \begin{align}  
 \<y, Db(x) y\> \le L_4 |y|^2,  \label{D-b}\\ 
\|D \sigma(x) y\|^2 +  |D \gamma(x, z) y|^2 
& \le L_5|y|^2.  \label{D-v}
\end{align}   
\end{rem}

Furthermore, the following assumption is required to derive a uniform moment bound for the BEM \eqref{bem} applied to the jump-free SODE \eqref{sde}.

\vspace{0.5em}

\begin{assumption} \label{A3*}
There exists a positive constant $L_4^*$ such that
\begin{align}
\label{coup-}
\<x-y,b(x)-b(y)\>\le -L_4^*|x-y|^2, \quad x,y\in\rr^d.
\end{align}
\end{assumption}

\begin{rem}
(1) Under the conditions \eqref{coup-} and \eqref{sg-lip}, the solution $X$ of jump-free Eq. \eqref{sde} possesses uniform bounded moments provided that  $2L_4^*>(2p-1)L_5$ and $\ee|X_0|^{2p} < \infty $ for some $p \in \nn_+$ (see \cite{WW25}):
 \begin{align*}
 \ee|X(t)|^{2p}\le C(1+\ee|X_0|^{2p}), \quad t \ge 0.
 \end{align*} 

(2) By \eqref{coup-} and \eqref{sg-lip} with $2L_4^*>(2p-1)L_5$, there exist  positive constants $L_7$, $L_8$, and $L_9=2L_4^*-(2p-1)L_5$ such that for any $x, y \in \rr^d$,
 \begin{align} 
 \<x,b(x)\> \le L_7-L_8|x|^2,  \label{coe-}  \\  
2\<y, Db(x)y\>+(2p-1) \|D \sigma(x) y\|^2 & \le -L_9|y|^2. \label{D-b+}
\end{align}
\end{rem}
 
 \vspace{-1em}

\subsection{Moment Estimates.} 
We begin with the following moment estimates of the BEM   \eqref{bem}. 
Here and what follows, we use $C$ to denote a generic constant independent of $n$ and $\tau$ that may differ in each appearance.

\vspace{0.5em}

\begin{lem} \label{lm-Y}
Assume $\ee|X_0|^{2p }<\infty$ with $p\in \nn_+$ and let Assumption \ref{A3} hold.
There exists a constant $C$ such that 
\begin{align} \label{Y-est}
\ee|Y_n|^{2p} \le e^{C t_n} (1+ \ee|X_0|^{2p}), \quad n \in \nn.
\end{align} 
Assume that $\gamma=0$ and the condition \eqref{coe-} holds for $L_8>L_6\cdot  C_p$ with 
$C_p=\max \{ \frac{2^{2p}-2p-1}{2p}, (\frac{2^{2p}-2p-1}{2^p}(2p-1)!!)^{1/p}\}.$
Then 
\begin{align}
\label{Y-est+}
\ee|Y_n|^{2p} \le C (1+ \ee|X_0|^{2p}), \quad n \in \nn.
\end{align} 
\end{lem} 
\begin{proof}
For $n \in \nn$ and $p\in \nn_+$, using the elementary inequality $|\beta+\gamma|^k\le 2^{k-1}(|\beta|^k+|\gamma|^k)$ for any $\beta, \gamma \in \rr$ and $k \in \nn_+$, we have
\begin{align*}
& |Y_{n+1}-b(Y_{n+1})\tau|^{2p} \\
&=|Y_n + 
    \sigma(Y_n ) \delta_nW +
    \delta_n \widetilde{N}(Y_n) |^{2p} \\
   & \le|Y_n|^{2p}+2p|Y_n|^{2p-2}\<Y_n,  \sigma(Y_n ) \delta_nW +   \delta_n \widetilde{N}(Y_n)\> \\
    &\quad+\sum_{k=2}^{2p}\binom{2p}{k}|Y_n|^{2p-k}|\delta_n \widetilde{N}(Y_n)+\sigma(Y_n ) \delta_nW|^k \\
    &\le |Y_n|^{2p}+2p|Y_n|^{2p-2}\<Y_n,  \sigma(Y_n ) \delta_nW +
    \delta_n \widetilde{N}(Y_n)\> \\
  &\quad+\sum_{k=2}^{2p}2^{k-1}\binom{2p}{k}|Y_n|^{2p-k}| \delta_n \widetilde{N}(Y_n)|^k
  +\sum_{k=2}^{2p}2^{k-1}\binom{2p}{k}|Y_n|^{2p-k}| \sigma(Y_n ) \delta_nW|^k.
\end{align*}
Taking the conditional expectation on both sides of the above inequality and using the condition \eqref{sg-grow}, we obtain
 \begin{align*}
\ee_n|Y_{n+1}-b(Y_{n+1})\tau|^{2p} &\le (1+C\tau)|Y_n|^{2p}+C\tau.
\end{align*}
According to the condition \eqref{coup} and the above estimate, we conclude \eqref{Y-est}.

To show \eqref{Y-est+}, testing \eqref{bem} with $Y_{n+1}$ and using the elementary equality
$\<\beta-\gamma, \gamma\>=\frac12 (|\beta|^2-|\gamma|^2) + \frac12 |\beta-\gamma|^2$ for all $\beta, \gamma \in \rr^d$,
we have
\begin{align*}
&\frac{1}{2}(|Y_{n+1}|^2-|Y_{n}|^2)+\frac{1}{2}|Y_{n+1}-Y_n|^2 \\
&=\<Y_{n+1},b(Y_{n+1})\>\tau+\<Y_{n+1}-Y_n,\sigma(Y_n)\delta_nW\>+\<Y_n,\sigma(Y_n)\delta_nW\>.
\end{align*}
By the condition \eqref{coe-}, we see
\begin{align*}
\frac{1}{2}(|Y_{n+1}|^2-|Y_{n}|^2)
\le  -L_8\tau|Y_{n+1}|^2+L_7\tau + \frac{1}{2}|\sigma(Y_n)\delta_nW|^2
+\<Y_n,\sigma(Y_n)\delta_nW\>,
\end{align*} 
from which we derive
\begin{align*}
(1+2L_8\tau) \ee_n |Y_{n+1}|^2
& \le  2L_7\tau + \ee_n |Y_n+\sigma(Y_n)\delta_nW|^2.
\end{align*}    
Taking the $p$-th power on both sides of the above inequality, taking conditional expectations, and using  Young inequality, we obtain
 \begin{align*}
(1+2L_8\tau)^p\ee_n|Y_{n+1}|^{2p}
 & \le \ee_n (2L_7\tau +|Y_n+\sigma(Y_n)\delta_nW|^2)^p \\
 & \le \ee_n|Y_n+\sigma(Y_n)\delta_nW|^{2p} 
 +C\tau |Y_n|^{2p-2}+C \tau^2.
 \end{align*}
It follows from \eqref{sg-grow} and  Young inequality that 
\begin{align*}
& \ee_n|Y_n+\sigma(Y_n)\delta_nW|^{2p}\\
 &\le |Y_n|^{2p}+2p\ee_n[|Y_n|^{2p-2}\<Y_n,\sigma(Y_n)\delta_nW\>]+
\sum_{j=2}^{2p}\binom{2p}{j} \ee_n[|Y_n|^{2p-j}|\sigma(Y_n)\delta_nW|^j] \\
 &= |Y_n|^{2p}+\sum_{j=0}^{2p-2}\binom{2p}{j+2} \ee_n[|Y_n|^{2p-2-j}|\sigma(Y_n)\delta_nW|^j|\sigma(Y_n)\delta_nW|^2] \\
&\le  |Y_n|^{2p}+\sum_{j=0}^{2p-2}\binom{2p}{j+2} \frac{2p-2-j}{2p-2}\ee_n[|Y_n|^{2p-2}|\sigma(Y_n)\delta_nW|^2] \\
&\quad+\sum_{j=0}^{2p-2}\binom{2p}{j+2} \frac{j}{2p-2}\ee_n|\sigma(Y_n)\delta_nW|^{2p}\\
& \le |Y_n|^{2p}+\sum_{j=0}^{2p-2}\binom{2p}{j+2}(L_6\tau+L_6^p(2p-1)!!\tau^{p}) |Y_n|^{2p} 
 +C\tau|Y_n|^{2p-2} \\
 &=|Y_n|^{2p}+(2^{2p}-2p-1)(L_6\tau+L_6^p(2p-1)!!\tau^{p}) |Y_n|^{2p} 
 +C\tau|Y_n|^{2p-2}.
\end{align*}

It is not difficult to show $(2^{2p}-2p-1)(L_6\tau+L_6^p(2p-1)!!\tau^{p})<\sum_{j=1}^p\binom{p}{j}(2L_8)^j\tau^j$ when $L_8>L_6 \cdot C_p$. Then taking $2\kappa:=\sum_{j=1}^p\binom{p}{j}(2L_8)^j\tau^{j-1}-(2^{2p}-2p-1)(L_6+L_6^p(2p-1)!!\tau^{p-1}) (1+\sum_{j=1}^p\binom{p}{j}(2L_8)^j\tau^j)^{-1}$ and using Young inequality, we have
\begin{align*}
\ee_n|Y_n|^{2p}
\le (1-2\kappa\tau)|Y_n|^{2p}+C\tau|Y_n|^{2p-2} +C \tau^2 
\le (1-\kappa\tau)|Y_n|^{2p}+C\tau.
\end{align*}
 By iteration, we derive \eqref{Y-est+}.
\end{proof}

\subsection{Regularity of Kolmogorov Equation.}
To obtain a weak error estimate in terms of 
$|\ee \varphi(X(T)-\ee \varphi(Y_N)|$,
for a class of test functions $\varphi:\rr^d\rightarrow \rr$, we introduce the function $u:[0,T] \times \rr^d \rightarrow \rr$ defined by
\begin{align}
\label{u-form}
u(t, x):=\ee \varphi(X_t^x),
\quad t \ge 0, ~ x \in \rr^d.
\end{align} 
According to \cite[Lemma 12.3.1]{PB10}, the function $u$ defined in \eqref{u-form} satisfies  the following Kolmogorov equation associated with Eq. \eqref{sde}:
\begin{align} \label{kol}
\left\{
\begin{aligned}
\partial_tu(t,x)&=\int_Z[
u(t,x+\gamma(x,z))-u(t,x) 
-Du(t,x)\gamma(x,z)] \nu(\dd z) \\
&\quad +Du(t,x)b(x)+\frac{1}{2}D^2u(t,x)\sigma^2(x) , \quad (t,x) \in (0, \infty) \times \rr^d, \\
u(0,x)&=\varphi(x), \quad x \in \rr^d.
\end{aligned}
\right.
\end{align}

To derive the required regularity of the Kolmogorov equation \eqref{kol}, we impose the non-degeneracy hypothesis of $\sigma$ and assume that the coefficients of Eq. \eqref{sde} have continuous partial derivatives up to third order, which grow polynomially. 
 
\vspace{0.5em}

 \begin{assumption} \label{A4}
 For each $x\in\rr^d$, $\sigma(x)\sigma(x)^{\tr}\in\rr^{d\times d}$ is positive definite.
Moreover, $b, \sigma_j, \gamma(\cdot, z) \in \CC^3(\rr^d)$, for all $j \in\{1,\cdots,m\}$ and  $z\in Z$, and there exist constants $C>0$ and $q \ge 3$ such that  
    \begin{align} \label{D3-v}
& | D^3 b(x)(v_1, v_2, v_3)| 
            + |D^3\sigma_j(x)(v_1,v_2,v_3)|
             + |D^3 \gamma(x,z)(v_1,v_2,v_3)| \nonumber\\
            & \le C (1+|x|)^{q-3} |v_1|  |v_2| |v_3|,
            \quad x,v_1,v_2,v_3 \in \rr^d.    
\end{align}
   \end{assumption}

\begin{rem}
 The existence and the uniqueness of the derivatives up to third order of the solution to Eq. \eqref{sde} can be derived similarly to shown in \cite[Section 1.3]{Cerrai01}, under Assumptions \ref{A3} and \ref{A4}. 
\end{rem}

\vspace{0.5em}

Before moving on,  we define a function  $\HH_{\cdot}(\cdot): \rr\times\rr^d\rightarrow[1,\infty)$: for $k \in \rr_+$ and $x\in\rr^d$, 
\begin{align*} 
\HH_k(x) := (1+|x|)^k,
\end{align*}
and introduce the following Gronwall-type inequality (see, e.g., \cite{GWW25}). 

\vspace{0.5em}

\begin{lem}
\label{lm-cite}
Let $h\in\rr$ and $f$ and $g$ be nonnegative continuous functions on $[h,\infty)$.  If there exists a positive constant $c$ such that
\begin{align*}
f(t)-f(s)\le -c\int_s^t f(u)\dd u+\int_s^tg(u)\dd u, \quad h\le s\le t <\infty,
\end{align*}
then
$f(t)\le f(h)+\int_h^t e^{-c(t-u)}g(u)\dd u$, $t \ge h$.
\end{lem}

\vspace{0.5em}

We have the following a priori estimates for the derivatives of the exact solution of Eq. \eqref{sde}.
For convenience, we denote
$|f|_{L_\omega^p}:=(\int_\Omega f(\omega) \pp(\dd \omega))^{1/p}$
for $f \in L^p(\Omega; \rr)$ or $f \in L^p(\Omega; \rr^d)$ with $p \ge 1$.

\vspace{0.5em}

\begin{lem} \label{lm-DX}
Let Assumptions \ref{A3} and \ref{A4} hold, $T>0$, $1 \le p_2 < p_1 \le p$, and $\rho_j > 0$ for $j \in \{1, \cdots, 6\}$ satisfy $\sum_{j=1}^3 1/ \rho_j=\sum_{j=4}^61/ \rho_j=1$.
 Then for any $t \in [0, T]$, there exists a constant $C > 0$ such that for any 
$v_1 \in L_\omega^{\max\{2p_1, 2\rho_2 p_2\}}$, 
$v_2 \in L_\omega^{\max\{2\rho_3 p_2, 2 \rho_3 \rho_5\}}$, and  
$v_3 \in L_\omega^{2 \rho_3 \rho_6}$,
\begin{align}
|\DD X_t^{\cdot} v_1 |_{L_\omega^{2p_1}} 
&\le  e^{C  t} |v_1|_{L_\omega^{2p_1}}, \label{DX-est} \\ 
| \DD^2 X_t^{\cdot}(v_1, v_2) |_{L_\omega^{2p_2}}
      &\le e^{C t}\sup_{s \in[0,T]}|\HH_{q-2}(X_s^\cdot)|_{L_\omega^{2\rho_1p_2}}
      |v_1|_{L_\omega^{2\rho_2p_2}}|v_2|_{L_\omega^{2\rho_3p_2}},\label{D2X-est} \\
|\DD^3 X_t^{\cdot}(v_1, v_2,v_3)|_{L_\omega^2}
&\le  e^{C  t}\sup_{s \in[0,T]}|\HH_{q-2}(X_s^\cdot)|_{L_\omega^{\max\{2\rho_1,2\rho_3\rho_4\}}}|v_1|_{L_\omega^{2\rho_2}} 
|v_2|_{L_\omega^{2\rho_3\rho_5}}|v_3|_{L_\omega^{2\rho_3\rho_6}}.
\label{D3X-est}
\end{align} 
Moreover, if $\gamma=0$ and Assumption \ref{A3*} holds with $2L_4^*>(2p-1)L_5$ for some $p \in \nn_+$,  there exist constants $C, \alpha>0$ such that
\begin{align}
|\DD X_t^{\cdot} v_1 |_{L_\omega^{2p_1}} 
&\le  e^{-\alpha t} |v_1|_{L_\omega^{2p_1}}, \label{DX-est+}\\
| \DD^2 X_t^{\cdot}(v_1, v_2) |_{L_\omega^{2p_1}}
&\le C e^{-\alpha t}\sup_{s\ge0}|\HH_{q-2}(X_s^\cdot)|_{L_\omega^{2\rho_1p_2}}
|v_1|_{L_\omega^{2\rho_2p_2}} |v_2|_{L_\omega^{2\rho_3p_2}},\label{D2X-est+} \\
|\DD^3 X_t^{\cdot}(v_1, v_2,v_3)|_{L_\omega^2}  
&\le  C e^{-\alpha t}\sup_{s\ge0} |\HH_{q-2}(X_s^\cdot)|_{L_\omega^{\max\{2\rho_1,2\rho_3\rho_4\}}}|v_1|_{L_\omega^{2\rho_2}}
|v_2|_{L_\omega^{2\rho_3\rho_5}}|v_3|_{L_\omega^{2\rho_3\rho_6}}. 
\label{D3X-est+}
\end{align}  
\end{lem}

\begin{proof} 
In what follows, for all $x, v_1, v_2, v_3\in \rr^d$, we denote 
\begin{align*}
    \eta^{v_1}(t, x) := \DD X_t^x v_1,  \quad 
    \xi^{v_1, v_2}(t, x) :=  \DD^2 X_t^x(v_1, v_2), \quad 
    \zeta^{v_1, v_2,v_3}(t, x) :=  \DD^3 X_t^x(v_1, v_2,v_3).
\end{align*} 
We omit the variable $x$ of $\eta$, $\xi$, and $\zeta$  for brevity if there is no confusion.

For $\dd\eta^{v_1}(t)$, we have $\eta^{v_1}(0)=v_1$ and 
\begin{align*}
    \dd\eta^{v_1}(t)&=Db(X_t)\eta^{v_1}(t) \dd t+\sum_{j=1}^mD\sigma_j(X_t)\eta^{v_1}(t)\dd W_j(t)
    + \int_Z D\gamma(X_t,z)\eta^{v_1}(t) \widetilde{N}(\dd t,\dd z).
\end{align*}
Define a stopping time as
\begin{align*}
\tau_n^1:=\inf\{s\in[0,t]:|\eta^{v_1}(s, \cdot)|>n \quad \mathrm{or} \quad |X_s^{\cdot}|>n\}.
\end{align*}
Applying It\^o formula (see \cite[Theorem 93]{Rong05}), and  the elementary inequality 
\begin{align*}
 |x+y|^{2r_1}-|x|^{2r_1}-2r_1|x|^{2r_1-2} \<x,y\>\le C(|x|^{2r_1}+|y|^{2r_1}),
\end{align*}
for any $x, y \in \rr^d$ and $r_1\in \nn_+$, we obtain 
\begin{align} \label{1var-Ito}
& \ee|\eta^{v_1}(t\wedge\tau_n^1)|^{2p_1} \nonumber  \\ 
&\le \ee|v_1|^{2p_1} + 2 p_1 \ee\int_{0}^{t\wedge\tau_n^1}
|\eta^{v_1}|^{2 p_1-2}\<\eta^{v_1},Db(X)\eta^{v_1}\>\dd s \nonumber \\
&\quad+p_1(2p_1-1) \ee \int_{0}^{t\wedge\tau_n^1}|\eta^{v_1}|^{2 p_1-2} \|D\sigma(X)\eta^{v_1}\|^2\dd s  \nonumber \\
&\quad+\ee\int_{0}^{t\wedge\tau_n^1}\int_Z[ |\eta^{v_1}+D\gamma(X,z)\eta^{v_1}|^{2p_1} -|\eta^{v_1}|^{2p_1}  \nonumber \\
&\quad-2p_1|\eta^{v_1}|^{2 p_1-2}\<\eta^{v_1},D\gamma(X,z)\eta^{v_1}\>] \nu(\dd z)\dd s \nonumber \\
&\le \ee|v_1|^{2p_1} + C \ee\int_{0}^{t\wedge\tau_n^1}|\eta^{v_1}|^{2p_1} \dd s+C\ee\int_{0}^{t\wedge\tau_n^1}\int_Z |D\gamma(X,z)\eta^{v_1}|^{2p_1} 
\nu(\dd z)\dd s \nonumber \\
& \quad + p_1 \ee \int_{0}^{t\wedge\tau_n^1} |\eta^{v_1}|^{2 p_1-2} 
\Big[ 2\<\eta^{v_1},Db(X)\eta^{v_1}\> + (2p_1-1) \|D\sigma(X)\eta^{v_1}\|^2 \Big] \dd s,
\end{align}
where we omit the integration variable here and after when there is an integration to lighten the notation.
Using the conditions \eqref{D-b} and \eqref{D-v}, we have
\begin{align*}
\ee|\eta^{v_1}(t\wedge\tau_n^1)|^{2p_1}\le\ee|v_1|^{2p_1} +C  \ee\int_{0}^{t\wedge\tau_n^1}|\eta^{v_1}|^{2p_1} \dd s.
\end{align*}
Using  Gronwall inequality and Fatou lemma, we obtain \eqref{DX-est}.

 Similarly, $\xi^{v_1,v_2}(t)$ satisfies $\xi^{v_1,v_2}(0)=0$ and 
\begin{align*}
\dd \xi^{v_1,v_2}(t)
&=[Db(X_t)\xi^{v_1,v_2}(t)+D^2b(X_t)(\eta^{v_1}(t),\eta^{v_2}(t))]\dd t \\
&\quad+\sum_{j=1}^m[D\sigma_j(X_t)\xi^{v_1,v_2}(t)+D^2\sigma_j(X_t)(\eta^{v_1}(t),\eta^{v_2}(t))]\dd W_j(t) \\
&\quad+\int_Z [D\gamma(X_t,z)\xi^{v_1,v_2}(t) 
+ D^2\gamma(X_t,z)(\eta^{v_1}(t),\eta^{v_2}(t))] \widetilde{N}(\dd t,\dd z).
\end{align*} 
Following the same idea as used to derive \eqref{1var-Ito} (and introducing an analogous stopping time if necessary), we have 
\begin{align} \label{2var-Ito}
& \ee|\xi^{v_1,v_2}(t)|^{2p_2} \nonumber  \\
&  \le 2p_2 \ee \int_{0}^{t}|\xi^{v_1,v_2}|^{2(p_2-1)}
\<\xi^{v_1,v_2},D^2b(X)(\eta^{v_1},\eta^{v_2})\>\dd s \nonumber \\
&\quad+p_2\ee\int_{0}^{t}|\xi^{v_1,v_2}|^{2(p_2-1)}[
2\<\xi^{v_1,v_2},Db(X)\xi^{v_1,v_2}\>
+(2p_2-1)(1+\epsilon) \|D\sigma(X)\xi^{v_1,v_2} \|^2]\dd s  \nonumber \\
&\quad+C\ee\int_{0}^{t}\int_Z|D\gamma(X,z)\xi^{v_1,v_2}|^{2p_2}\nu(\dd z)\dd s+C\ee\int_{0}^{t}|\xi^{v_1,v_2}|^{2p_2}\dd s\nonumber \\
&\quad + C \ee \int_{0}^{t} \|D^2\sigma(X)(\eta^{v_1},\eta^{v_2})\|^{2p_2}\dd s
+C\ee\int_{0}^{t}\int_Z |D^2\gamma(X,z)(\eta^{v_1},\eta^{v_2})|^{2p_2} \nu(\dd z)\dd s \nonumber\\
&\le C\ee\int_{0}^{t}|\xi^{v_1,v_2}|^{2p_2}\dd s + C\ee\int_{0}^{t}
|D^2b(X)(\eta^{v_1},\eta^{v_2})|^{2p_2}\dd s \nonumber \\
&\quad+C \ee\int_{0}^{t} \|D^2\sigma(X)(\eta^{v_1},\eta^{v_2})\|^{2p_2}\dd s 
+C\ee\int_{0}^{t}
 \int_Z
|D^2\gamma(X,z)(\eta^{v_1},\eta^{v_2})|^{2p_2}\nu(\dd z)\dd s \nonumber \\
&=:C \ee\int_{0}^{t}|\xi^{v_1,v_2}|^{2p_2}\dd s+I_{11}+I_{12}+I_{13}.
\end{align} 
For the term $I_{11}$, by using \eqref{D3-v} and \eqref{DX-est},  we can show that
\begin{align*}
I_{11} 
&\le C \ee\int_{0}^{t} (1+|X|)^{2p_2(q-2)}|\eta^{v_1}|^{2p_2}|\eta^{v_2}|^{2p_2}\dd s \\
&\le Ce^{Ct}\sup_{s\in[0,T]}|\HH_{q-2}(X_s)|^{2p_2}_{L_\omega^{2\rho_1p_2}}|v_1|^{2p_2}_{L_\omega^{2\rho_2p_2}}|v_2|^{2p_2}_{L_\omega^{2\rho_3p_2}},
\end{align*}
for some positive constants $\rho_i$ satisfying $\sum_{i=1}^31/\rho_i=1$.
For the terms $I_{12}$ and $I_{13}$, by using Assumption \ref{A4} and \eqref{DX-est},  we obtain
\begin{align*}
I_{12} +I_{13}\le e^{Ct}\sup_{s\in[0,T]}|\HH_{q-2}(X_s)|^{2p_2}_{L_\omega^{2\rho_1p_2}}|v_1|^{2p_2}_{L_\omega^{2\rho_2p_2}}|v_2|^{2p_2}_{L_\omega^{2\rho_3p_2}}. 
\end{align*} 
We derive \eqref{D2X-est} by the above two estimates, \eqref{2var-Ito}, and Gronwall inequality.

For the last term $\zeta^{v_1,v_2,v_3}(t)$, we get $\zeta^{v_1,v_2,v_3}(0,x)=0$ and 
\begin{align*}
\dd \zeta^{v_1,v_2,v_3}(t) 
&=[Db(X_t)\zeta^{v_1,v_2,v_3}(t)+D^2b(X_t)(\xi^{v_2,v_3}(t),\eta^{v_1}(t))\\
&\qquad+D^2b(X_t)(\xi^{v_1,v_3}(t),\eta^{v_2}(t))+D^2(X_t)(\xi^{v_1,v_2}(t),\eta^{v_3}(t))\\
&\qquad + D^3b(X_t)(\eta^{v_1}(t),\eta^{v_2}(t),\eta^{v_3}(t))] \dd t \\
&\quad + \sum_{j=1}^m [D\sigma_j(X_t)\zeta^{v_1,v_2,v_3}(t)+D^2\sigma_j(X_t)(\xi^{v_2,v_3}(t),\eta^{v_1}(t)) \\
&\qquad+D^2\sigma_j(X_t)(\xi^{v_1,v_3}(t),\eta^{v_2}(t))
+D^2\sigma_j(X_t)(\xi^{v_1,v_2}(t),\eta^{v_3}(t))\\
&\qquad+D^3\sigma_j(X_t)(\eta^{v_1}(t),\eta^{v_2}(t),\eta^{v_3}(t))] \dd W_j(t) \\
&\quad + \int_Z [D^3\gamma(X_t,z) \zeta^{v_1,v_2,v_3}(t)+D^2\gamma(X_t,z)(\xi^{v_2,v_3}(t),\eta^{v_1}(t)) \\
&\qquad +D^2\gamma(X_t,z)(\xi^{v_1,v_3}(t),\eta^{v_2}(t))
 +D^2\gamma(X_t,z)(\xi^{v_1,v_2}(t),\eta^{v_3}(t))\\
&\qquad+D^3\gamma(X_t,z)(\eta^{v_1}(t),\eta^{v_2}(t),\eta^{v_3}(t)))]\widetilde{N}(\dd t,\dd z).
\end{align*} 
Then
\begin{align} \label{3var-Ito}
& \ee|\zeta^{v_1,v_2,v_3}(t)|^2  \nonumber \\
& \le C \ee\int_{0}^{t} |\zeta^{v_1,v_2,v_3}|^2 \dd s 
+ C \ee\int_{0}^{t} [|D^2b(X)(\xi^{v_2,v_3},\eta^{v_1})|^2
+|D^2b(X)(\xi^{v_1,v_3},\eta^{v_2})|^2  \nonumber \\
&\qquad +  |D^2b(X)(\xi^{v_1,v_2},\eta^{v_3})|^2 
+|D^3b(X)(\eta^{v_1},\eta^{v_2},\eta^{v_3})|^2] \dd s \nonumber \\
&\quad + C \ee\int_{0}^{t} [\|D^2\sigma(X)(\xi^{v_2,v_3},\eta^{v_1})\|^2 + 
\|D^2\sigma(X)(\xi^{v_1,v_3},\eta^{v_2})\|^2 \nonumber \\
&\qquad + \|D^2\sigma(X)(\xi^{v_1,v_2},\eta^{v_3})\|^2
+ \|D^3\sigma(X)(\eta^{v_1},\eta^{v_2},\eta^{v_3})\|^2] \dd s
\nonumber \\
&\quad+C \ee\int_{0}^{t}\int_Z [|D^2\gamma(X,z)(\xi^{v_2,v_3},\eta^{v_1})|^2
+ |D^2\gamma(X,z)(\xi^{v_1,v_3},\eta^{v_2})|^2  \nonumber \\
&\quad + |D^2\gamma(X,z)(\xi^{v_1,v_2},\eta^{v_3})|^2
+ |D^3\gamma(X,z)(\eta^{v_1},\eta^{v_2},\eta^{v_3})|^2] \nu(\dd z)\dd s. \nonumber \\
&=:C \ee\int_{0}^{t}|\zeta^{v_1,v_2,v_3}|^2\dd s+I_{21}+I_{22}+I_{23}.
\end{align} 
We begin with the estimation of the term $I_{21}$ as follows:
\begin{align}
\label{I21}
I_{21} &\le C\ee\int_{0}^{t} (1+|X|)^{2q-4}|\xi^{v_2,v_3}|^2|\eta^{v_1}|^2 \dd s
+C\ee\int_{0}^{t} (1+|X|)^{2q-4}|\xi^{v_1,v_3}|^2|\eta^{v_2}|^2 \dd s \nonumber \\
&\quad+C\ee\int_{0}^{t} (1+|X|)^{2q-4}|\xi^{v_1,v_2}|^2|\eta^{v_3}|^2 \dd s 
+C\ee\int_{0}^{t} (1+|X|)^{2q-6}|\eta^{v_1}|^2|\eta^{v_2}|^2|\eta^{v_3}|^2
\dd s \nonumber \\
&=:I_{21}^1+I_{21}^2+I_{21}^3+I_{21}^4,
\end{align}
where we used the condition \eqref{D3-v}.
We note that the analysis of $I_{21}^1,I_{21}^2,I_{21}^3$ are the same, and we take $I_{21}^1$ and $I_{21}^2$ as examples.
By   \eqref{DX-est},  \eqref{D2X-est}, and the H\"older inequality with $\rho_j$ satisfying
$\sum_{j=1}^31/ \rho_j=\sum_{j=4}^6 1/\rho_j=1$, we have 
\begin{align*}
I_{21}^1 &\le C\int_0^t|\HH_{q-2}(X)|^2_{L_\omega^{2\rho_1}}|\eta^{v_1}|^2_{L_\omega^{2\rho_2}}|\xi^{v_2,v_3}|^2_{L_\omega^{2\rho_3}}\dd s \\
&\le Ce^{Ct}\sup_{s\in[0,T]}|\HH_{q-2}(X_s)|^2_{L_\omega^{\max\{2\rho_1 ,2\rho_3 \rho_4\}}}|v_1|^2_{L_\omega^{2\rho_2}}
 |v_2|^2_{L_\omega^{2\rho_3\rho_5}}|v_3|^2_{L_\omega^{2\rho_3\rho_6}}.
\end{align*}
For  $I_{21}^2$, we choose others $\chi_j>0$ with
$\sum_{j=1}^3 1/\chi_j=\sum_{j=4}^6 1/\chi_j=1$ to show
\begin{align*}
I_{21}^2 &\le C\int_0^t|\HH_{q-2}(X)|^2_{L_\omega^{2\chi_1}}|\eta^{v_2}|^2_{L_\omega^{2\chi_2}}|\xi^{v_1,v_3}|^2_{L_\omega^{2\chi_3}}\dd s \\
&\le Ce^{Ct}\sup_{s\in[0,T]}|\HH_{q-2}(X_s)|^2_{L_\omega^{\max\{2\chi_1,2\chi_3\chi_4\}}}|v_1|^2_{L_\omega^{2\chi_3\chi_5}} 
|v_2|^2_{L_\omega^{2\chi_2}}|v_3|^2_{L_\omega^{2\chi_3\chi_6}}.
\end{align*}
Then we take $\chi_1=\rho_1, \chi_3\chi_4=\rho_3\rho_4,\chi_3\chi_5=\rho_2,\chi_2=\rho_3\rho_5,\chi_3\chi_6=\rho_3\rho_6$. It is obvious that
$\tfrac{1}{\rho_3}=\tfrac{1}{\chi_3\chi_4}+\tfrac{1}{\chi_2}+\tfrac{1}{\chi_3\chi_6}=1-\tfrac{1}{\rho_1}-\tfrac{1}{\rho_2}=1-\tfrac{1}{\chi_1}-\tfrac{1}{\chi_3\chi_5},$
which also leads to
$\tfrac{1}{\chi_1}+\tfrac{1}{\chi_2}+\tfrac{1}{\chi_3\chi_4}+\tfrac{1}{\chi_3\chi_5}+\tfrac{1}{\chi_3\chi_6}=\tfrac{1}{\chi_1}+\tfrac{1}{\chi_2}+\tfrac{1}{\chi_3}(\tfrac{1}{\chi_4}+\tfrac{1}{\chi_5}+\tfrac{1}{\chi_6})=1.$
This implies a one-to-one correspondence between $\rho_i$ and $\chi_i$ for $i\in \{1,2,3,4,5,6\}$. 

For $I_{21}^4$, we apply  \eqref{DX-est}   and the H\"older inequality with positive constants  $\rho'_j$ satisfying $\sum_{j=1}^41/\rho'_j=1$ to obtain
\begin{align*}
I_{21}^4 &\le  Ce^{Ct}\sup_{s\in[0,T]}|\HH_{q-3}(X_s)|^2_{L_\omega^{2\rho'_1}}|v_1|^2_{L_\omega^{2\rho'_2}}|v_2|^2_{L_\omega^{2\rho'_3}}|v_3|^2_{L_\omega^{2\rho'_4}}. 
\end{align*}
Assuming, for example, $\rho'_1=\rho_1$ and $\rho'_2=\rho_2$, we can choose
$\rho'_3$ and $\rho'_4$ such that $\rho_3\rho_5\ge\rho'_3$ and $\rho_3\rho_6\ge\rho'_4$. That is, $I_{21}^4$ can be  controlled by $I_{21}^1$. Hence, we obtain
\begin{align*}
I_{21}&\le Ce^{Ct}\sup_{s\in[0,T]}|\HH_{q-2}(X_s)|^2_{L_\omega^{\max\{2\rho_1 ,2\rho_3 \rho_4\}}}|v_1|^2_{L_\omega^{2\rho_2}} |v_2|^2_{L_\omega^{2\rho_3\rho_5}}|v_3|^2_{L_\omega^{2\rho_3\rho_6}}.
\end{align*}
Similarly, it is not hard to get
\begin{align*}
 I_{22}+I_{23}&\le Ce^{Ct}\sup_{s\in[0,T]}|\HH_{q-2}(X_s)|^2_{L_\omega^{\max\{2\rho_1 ,2\rho_3 \rho_4\}}}|v_1|^2_{L_\omega^{2\rho_2}} |v_2|^2_{L_\omega^{2\rho_3\rho_5}}|v_3|^2_{L_\omega^{2\rho_3\rho_6}}.
\end{align*}
Then \eqref{3var-Ito} can be bounded as
\begin{align*}
& \ee|\zeta^{v_1,v_2,v_3}(t)|^2 \\
&\le C \ee\int_{0}^{t}|\zeta^{v_1,v_2,v_3}|^2\dd s 
+ Ce^{Ct}\sup_{s\in[0,T]}|\HH_{q-2}(X_s)|^2_{L_\omega^{\max\{2\rho_1 ,2\rho_3 \rho_4\}}}
 |v_1|^2_{L_\omega^{2\rho_2}}|v_2|^2_{L_\omega^{2\rho_3\rho_5}}|v_3|^2_{L_\omega^{2\rho_3\rho_6}}.
\end{align*}
By Gronwall inequality, we obtain \eqref{D3X-est}. 

In the case $\gamma=0$, by the condition \eqref{D-b+}, the first process $\eta^{v_1}(t)$ satisfies 
\begin{align*}
& \ee e^{p_1L_9 t}|\eta^{v_1}(t)|^{2p_1}  \\
&\le \ee|v_1|^{2p_1} + p_1L_9\ee\int_{0}^{t}e^{p_1\alpha s}|\eta^{v_1}|^{2p_1} \dd s \\
&\quad+p_1\ee\int_{0}^{t}e^{p_1\alpha s}
|\eta^{v_1}|^{2 p_1-2}[2\<\eta^{v_1},Db(X)\eta^{v_1}\>
+ (2p_1-1) \|D\sigma(X)\eta^{v_1}\|^2]\dd s \\
&\le\ee|v_1|^{2p_1},
\end{align*}
which shows \eqref{DX-est+}.
The second process $\xi^{v_1,v_2}(t)$ satisfies
\begin{align} \label{2var-Ito+}
& \ee |\xi^{v_1,v_2}(t)|^{2p_2}  \nonumber \\
&\le 2p_2\ee\int_{0}^{t}|\xi^{v_1,v_2}|^{2(p_2-1)}
\<\xi^{v_1,v_2},D^2b(X)(\eta^{v_1},\eta^{v_2})\>\dd s \nonumber \\
&\quad+p_2\ee\int_{0}^{t}|\xi^{v_1,v_2}|^{2(p_2-1)}
[ 2\<\xi^{v_1,v_2},Db(X)\xi^{v_1,v_2}\>
+(2p_2-1)(1+\epsilon_1) \|D\sigma(X) \xi^{v_1,v_2} \|^2]\dd s  \nonumber \\
&\quad+C_{\epsilon_1} \ee\int_{0}^{t} |\xi^{v_1,v_2}|^{2(p_2-1)} \|D^2\sigma(X)(\eta^{v_1},\eta^{v_2})\|^{2}\dd s
\nonumber\\
&\le -(p_2L_9-\epsilon_2)\ee\int_{0}^{t}|\xi^{v_1,v_2}|^{2p_2}\dd s \nonumber + C_{\epsilon_2}\ee\int_{0}^{t}
|D^2b(X)(\eta^{v_1},\eta^{v_2})|^{2p_2}\dd s \nonumber \\
&\quad+C_{\epsilon_1,\epsilon_2}\ee\int_{0}^{t} \|D^2\sigma(X)(\eta^{v_1},\eta^{v_2})\|^{2p_2}\dd s \nonumber \\
&=:-(p_2L_9-\epsilon_2) \ee\int_{0}^{t}|\xi^{v_1,v_2}|^{2p_2}\dd s+I_{11}+I_{12},
\end{align}
for a positive constant $\epsilon_2 \in (0,p_2L_9)$.  
For $I_{11}$, by using \eqref{D3-v} and \eqref{DX-est},  we have 
\begin{align*}
 I_{11}  
&\le C \Big(\int_{0}^{t}e^{-4p_2\alpha s}\dd s\Big) \sup_{s\ge0}|\HH_{q-2}(X_s)|^{2p_2}_{L_\omega^{2\rho_1p_2}}  |v_1|^{2p_2}_{L_\omega^{2\rho_2p_2}} |v_2|^{2p_2}_{L_\omega^{2\rho_3p_2}},
\end{align*}
for some positive constants $\rho_j$ satisfying $\sum_{j=1}^31/\rho_j=1$.
For the term $I_{12}$, by using Assumption \ref{A4} and \eqref{DX-est+},  we obtain
\begin{align*}
I_{12} 
&\le C \Big(\int_{0}^{t}e^{-4p_2\alpha s}\dd s\Big) \sup_{s\ge 0} |\HH_{q-2}(X_s)|^{2p_2}_{L_\omega^{2\rho_1p_2}}  |v_1|^{2p_2}_{L_\omega^{2\rho_2p_2}}|v_2|^{2p_2}_{L_\omega^{2\rho_3p_2}}. 
\end{align*}
Combining the above two estimates and \eqref{2var-Ito+},  we have
\begin{align*}
& \ee|\xi^{v_1,v_2}(t)|^{2p_2}+ (p_2L_9-\epsilon_2) \ee\int_{0}^{t}|\xi^{v_1,v_2}|^{2p_2}\dd s \\
&\le  C \Big(\int_{0}^{t}e^{-4p_2\alpha s}\dd s\Big) \sup_{s\ge 0} |\HH_{q-2}(X_s)|^{2p_2}_{L_\omega^{2\rho_1p_2}}|v_1|^{2p_2}_{L_\omega^{2\rho_2p_2}}|v_2|^{2p_2}_{L_\omega^{2\rho_3p_2}}. 
\end{align*}
By  Lemma \ref{lm-cite}, we derive \eqref{D2X-est+}. 

Similarly, the last process $\zeta^{v_1,v_2,v_3}(t)$ satisfies 
\begin{align} \label{3var-Ito+}
\ee|\zeta^{v_1,v_2,v_3}(t)|^2 
&\le -(\alpha- \epsilon_4) \ee\int_{0}^{t}|\zeta^{v_1,v_2,v_3}|^2\dd s \nonumber \\ 
&\quad + C_{\epsilon_4} \ee\int_{0}^{t}(|D^2b(X)(\xi^{v_2,v_3},\eta^{v_1})|^2
+|D^2b(X)(\xi^{v_1,v_3},\eta^{v_2})|^2\nonumber \\
&\qquad + |D^2b(X)(\xi^{v_1,v_2},\eta^{v_3})|^2 
+ |D^3b(X)(\eta^{v_1},\eta^{v_2},\eta^{v_3})|^2
)\dd s \nonumber \\
&\quad+C_{\epsilon_3,\epsilon_4} \ee \int_{0}^{t}(\|D^2\sigma(X)(\xi^{v_2,v_3},\eta^{v_1})\|^2 
+ \|D^2\sigma(X)(\xi^{v_1,v_3},\eta^{v_2})\|^2
\nonumber \\
&\qquad+ \|D^2\sigma(X)(\xi^{v_1,v_2},\eta^{v_3})\|^2
+\|D^3\sigma(X)(\eta^{v_1},\eta^{v_2},\eta^{v_3})\|^2)\dd s
\nonumber \\
&=: -(L_9- \epsilon_4)\ee\int_{0}^{t}|\zeta^{v_1,v_2,v_3}|^2\dd s+I_{21}+I_{22},
\end{align}
for positive constants $\epsilon_3\in(0,2p-1)$ and $\epsilon_4\in(0,L_9)$. 
For  \eqref{I21}, it is not hard to  show that the analysis of $I_{21}^1, I_{21}^2, I_{21}^3$ are equivalent. We choose some positive constants $\rho_j$ satisfying $\sum_{j=1}^31/\rho_j=\sum_{j=4}^61/\rho_j =1$ to get
\begin{align*}
I_{21}^1 
&\le C\int_{0}^{t}|\HH_{q-2}(X)|^2_{L_\omega^{2\rho_1}}|\eta^{v_1}|^2_{L_\omega^{2\rho_2}} |\xi^{v_2,v_3}|^2_{L_\omega^{2\rho_3}} \dd s\\
&\le C \Big(\int_{0}^{t} e^{-4\alpha s} \dd s\Big)\sup_{s\ge 0} |\HH_{q-2}(X_s)|^2_{L_\omega^{\max\{2\rho_1, 2\rho_3\rho_4\}}}
|v_1|^2_{L_\omega^{2\rho_2}} |v_2|^2_{L_\omega^{2\rho_3\rho_5}}|v_3|^2_{L_\omega^{2\rho_3\rho_6}}.
\end{align*}
For $I_{21}^4$, we apply   \eqref{DX-est},   and the H\"older inequality with positive constants $\rho'_j$ satisfying $\sum_{j=1}^41/\rho'_j=1$ to obtain
\begin{align*}
I_{21}^4 &\le  C\Big(\int_{0}^{t} e^{-6\alpha s} \dd s\Big) \sup_{s\ge 0} |\HH_{q-3}(X_s)|^2_{L_\omega^{2\rho'_1}}|v_1|^2_{L_\omega^{2\rho'_2}}|v_2|^2_{L_\omega^{2\rho'_3}}|v_3|^2_{L_\omega^{2\rho'_4}}. 
\end{align*}
Assuming, for example, $\rho'_1=\rho_1$ and $\rho'_2=\rho_2$, we can choose $\rho'_3$ and $\rho'_4$ such that $\rho_3\rho_5\ge\rho'_3$ and $\rho_3\rho_6\ge\rho'_4$ . Hence, we have
\begin{align*}
I_{21} &\le C \Big(\int_{0}^{t} e^{-4\alpha s} \dd s\Big) \sup_{s\ge 0} |\HH_{q-2}(X_s)|^2_{L_\omega^{\max\{2\rho_1, 2\rho_3\rho_4\}}}  |v_1|^2_{L_\omega^{2\rho_2}}|v_2|^2_{L_\omega^{2\rho_3\rho_5}}|v_3|^2_{L_\omega^{2\rho_3\rho_6}}.
\end{align*}
Similarly, it is not hard to get
\begin{align*}
 I_{22} &\le C \Big(\int_{0}^{t} e^{-4\alpha s} \dd s\Big) \sup_{s\ge 0} |\HH_{q-2}(X_s)|^2_{L_\omega^{\max\{2\rho_1, 2\rho_3\rho_4\}}} |v_1|^2_{L_\omega^{2\rho_2}}|v_2|^2_{L_\omega^{2\rho_3\rho_5}}|v_3|^2_{L_\omega^{2\rho_3\rho_6}}.
\end{align*}
Plugging these estimates with \eqref{3var-Ito+}, we have
\begin{align*}
&\ee|\zeta^{v_1,v_2,v_3}(t)|^2 +(L_9- \epsilon_4)\ee\int_{0}^{t}|\zeta^{v_1,v_2,v_3}|^2\dd s\\
&\le C \Big(\int_{0}^{t} e^{- 4\alpha s} \dd s\Big) \sup_{s\ge 0} |\HH_{q-2}(X_s)|^2_{L_\omega^{\max\{2\rho_1, 2\rho_3\rho_4\}}}  |v_1|^2_{L_\omega^{2\rho_2}}|v_2|^2_{L_\omega^{2\rho_3\rho_5}}|v_3|^2_{L_\omega^{2\rho_3\rho_6}}.
\end{align*}
By  Lemma \ref{lm-cite}, we obtain \eqref{D3X-est+}.
\end{proof}

\vspace{0.5em}

As a byproduct of Lemma \ref{lm-DX} and the chain rule, we have the following estimate of the derivatives of the solution to the Kolmogorov equation \eqref{kol}.

\vspace{0.5em}

\begin{cor} \label{cor-Du} 
Let $\varphi\in C_b^3(\rr^d)$, Assumptions \ref{A3} and \ref{A4} hold, $T>0$, and $\rho_j > 0$ for $j \in \{1, \cdots, 6\}$ satisfy $\sum_{j=1}^3 1/ \rho_j=\sum_{j=4}^61/ \rho_j=1$.
 Then for any $t \in [0, T]$, $u(t, \cdot)\in C_b^3 (\rr^d)$ a.s. and there exists a constant $C > 0$ such that for any $v_1 \in L_\omega^{2\rho_2 }$, $v_2 \in L_\omega^{2 \rho_3 \rho_5}$, and 
$v_3 \in L_\omega^{2 \rho_3 \rho_6}$, 
    \begin{align}
     |D u(t, \cdot) v_1|_{L_\omega^1} &\le
       e^{C t}|v_1|_{L_\omega^2},     \label{Du-est}\\
       |D^2 u(t, \cdot)(v_1, v_2)|_{L_\omega^1} & \le e^{C t} \sup_{s\in[0,T]}|\HH_{q-2}(X_s^\cdot)|_{L_\omega^{2\rho_1}}|v_1|_{L_\omega^{2\rho_2}}|v_2|_{L_\omega^{2\rho_3}} \label{D2u-est},  \\
    \label{D3u-est}
 |D^3u(t, \cdot)(v_1,v_2,v_3)|_{L_\omega^1} 
 &\le e^{C  t} \sup_{s\in[0,T]}|\HH_{q-2}(X_s^\cdot)|_{L_\omega^{\max\{2\rho_1,2\rho_3\rho_4\}}}|v_1|_{L_\omega^{2\rho_2}}|v_2|_{L_\omega^{2\rho_3\rho_5}}|v_3|_{L_\omega^{2\rho_3\rho_6}}.
 \end{align} 
Moreover, if $\gamma=0$ and Assumption \ref{A3*} holds with $2L_4^*>(2p-1)L_5$ for some $p \in \nn_+$,  there exist constants $C, \alpha>0$ that  
    \begin{align}
    |D u(t, \cdot) v_1|_{L_\omega^1} &\le 
      C e^{-\alpha t} |v_1|_{L_\omega^2},\label{Du-est+} \\ 
  |D^2 u(t, \cdot)(v_1, v_2)|_{L_\omega^1} 
     &\le C e^{-\alpha t} \sup_{s\ge 0}|\HH_{q-2}(X_s^\cdot)|_{L_\omega^{2\rho_1}}|v_1|_{L_\omega^{2\rho_2}} |v_2|_{L_\omega^{2\rho_3}},         \label{D2u-est+} \\
    \label{D3u-est+}
  |D^3u(t, \cdot)(v_1,v_2,v_3)|_{L_\omega^1}
&\le C e^{-\alpha t} \sup_{s\ge 0}|\HH_{q-2}(X_s^\cdot)|_{L_\omega^{\max\{2\rho_1,2\rho_3\rho_4\}}}|v_1|_{L_\omega^{2\rho_2}}|v_2|_{L_\omega^{2\rho_3\rho_5}}|v_3|_{L_\omega^{2\rho_3\rho_6}}.
 \end{align}   
  \end{cor}

Before proceeding further, note that the BEM  \eqref{bem} is not necessarily continuous throughout the entire time interval.
 To overcome the potential implicitness and discrete-time nature  of the BEM  \eqref{bem}, similar to \cite{PWW24}, 
we introduce its continuous-time interpolation $\{\yy(t)\}_{t\in [t_n,t_{n+1}]}$ and $n \in \{0,1,\cdots,N-1\}$:
\begin{align}
\yy(t)&=\yy(t_n)+b(Y_n)(t-t_n)+\sigma(Y_n)(W(t)-W(t_n))+\int_{t_n}^t\int_Z\gamma(Y_n,z)\widetilde{N}(\dd s, \dd z), \nonumber \\
\yy(t_n)&=Y_n- b(Y_n)\tau.   \label{yy} 
\end{align}
It is straightforward to verify that $\yy(t_{n+1})=Y_{n+1}- b(Y_{n+1})\tau$, and for any $s,t\in [t_n,t_{n+1}]$,  the left limit $\lim\limits_{s \rightarrow t^{-}} \yy(s)$ exists and $\yy(t)=\lim\limits_{s \rightarrow t^{+}} \yy(s)$ holds.  Thus $\yy$ is a c\`adl\`ag process.  

We have the following estimates of the continuous interpolation $\yy$.

\vspace{0.5em}

\begin{lem} \label{lm-yy} 
Assume $Y_n \in L^{2qq_2}(\Omega)$ with $q_2\in \nn_+$ and let Assumption \ref{A3} hold.
For any $t_n \le s \le t \le t_{n+1}$ with $n \in \nn$, there exists a constant $C > 0$ such that
\begin{align*}
|\yy(t)|_{L_\omega^{2q_2}} & \le C|Y_n|^{q}_{L_\omega^{2qq_2}},\quad
|\yy(t)-\yy(s)|_{L_\omega^{2q_2}} \le C(t-s)^{1/(2q_2)}|Y_n|^{q}_{L_\omega^{2qq_2}} .
\end{align*}
If $\gamma=0$ and Assumption \ref{A3*} holds with $2L_4^*>(2p-1)L_5$ for some $p \in \nn_+$, then
\begin{align*}
|\yy(t)|_{L_\omega^{2q_2}} & \le C|Y_n|^{q}_{L_\omega^{2qq_2}},\quad 
|\yy(t)-\yy(s)|_{L_\omega^{2q_2}} \le C(t-s)^{1/2}|Y_n|^{q}_{L_\omega^{2qq_2}} .
\end{align*} 
\end{lem}

\begin{proof}
For the  jump-diffusion SODE \eqref{sde}, by \eqref{yy},  triangle inequality, and the condition \eqref{D3-v}, for $n\in \nn$, $q_2 \in \nn_+$,  and $t,s  \in [t_n,t_{n+1}]$,  we have
\begin{align*}
|\yy(t)|_{L_\omega^{2q_2}} &\le |\yy(t_n)|_{L_\omega^{2q_2}}+(t-t_n)|b(Y_n)|_{L_\omega^{2q_2}}  + (t-t_n)^{\frac{1}{2}}\|\sigma(Y_n)\|_{L_\omega^{2q_2}} \\
&\quad+(t-t_n)^{\frac{1}{2q_2}}\int_Z|\gamma(Y_n,z)|_{L_\omega^{2q_2}}\nu(\dd z) \le C|Y_n|^q_{L_\omega^{2qq_2}},
\end{align*}
and
\begin{align*}
|\yy(t)-\yy(s)|_{L_\omega^{2q_2}} &\le (t-s)|b(Y_n)|_{L_\omega^{2q_2}}+ (t-s)^{\frac{1}{2}}\|\sigma(Y_n)\|_{L_\omega^{2q_2}} \\
&\quad+(t-s)^{\frac{1}{2q_2}}\int_Z|\gamma(Y_n,z)|_{L_\omega^{2q_2}}\nu(\dd z)  \le C(t-s)^{\frac{1}{2q_2}}|Y_n|^q_{L_\omega^{2qq_2}} .
\end{align*}
This shows the first part.
Another part can be shown analogously for the jump-free SODEs \eqref{sde}.
\end{proof}

\section{Weak Convergence Analysis of BEM}
\label{sec4}

\vspace{0.5em}

In this section, we pursue two primary objectives.
Firstly, we establish a quantitative weak error estimate between the BEM \eqref{bem} and the jump-diffusion SODE \eqref{sde} in the finite time interval $[0, T]$, where $T$ is a given positive number.  
Secondly, we derive an infinite-time weak error analysis for the BEM  \eqref{bem} applied to the jump-free Eq. \eqref{sde}, providing an error estimate for corresponding continuous and discrete invariant measures.

In what follows, we will use the deterministic Taylor formula frequently for a differentiable function $f:\rr^d\rightarrow\rr^d$:
$f(x)-f(y)=Df(y)(x-y)+\RR_f(x,y),$
where
\begin{align*}
\RR_f(x,y):=\int_0^1(Df(y+r(x-y))-Df(y))(x-y)\dd r, \quad x,y\in \rr^d.
\end{align*}
To streamline subsequent analysis,  we define the following terms:
\begin{align*}
J_1^n&:=\ee \int_{t_n}^{t_{n+1}} Du(T-s,\yy)(b(Y_n)-b(\yy)) \dd s, \\
J_2^n&:=\frac{1}{2}\ee \int_{t_n}^{t_{n+1}} D^2u(T-s,\yy)
(\sigma_j(Y_n),\sigma_j(Y_n))
-D^2u(T-s,\yy)(\sigma_j(\yy),\sigma_j(\yy))\dd s,  \\
J_3^n&:=\ee \int_{t_n}^{t_{n+1}} \int_Z u(T-s,\yy+\gamma(Y_n,z))-u(T-s,\yy+\gamma(\yy,z))\nu(\dd z)\dd s,\\
J_4^n&:=\ee \int_{t_n}^{t_{n+1}} \int_Z Du(T-s,\yy)(\gamma(\yy,z)-\gamma(Y_n,z))\nu(\dd z)\dd s. 
\end{align*}
Next, we successively estimate these four terms.

Let us begin with the estimates of the first two terms above.

\vspace{0.5em}

\begin{lem} \label{lm-J1}
Assume $X_0\in L^{(4q-2)q}(\Omega)$ and let Assumptions \ref{A3} and \ref{A4} hold.   
There exists a constant $C > 0$ such that
\begin{align}
\label{J1}
\sum_{n=0}^{N-1}|J_1^n|  \le  e^{C T} (1+|X_0|^{(2q-1)q}_{L_\omega^{(4q-2)q}})\tau^{\frac{1}{2q-1}}.
\end{align}
Moreover, if $\gamma=0$ and Assumption \ref{A3*} holds with $2L_4^*>(2p-1)L_5$ and  $L_8>L_6\cdot C_p$ for some $p \in \nn_+$, where $C_p$ is the constant in Lemma \ref{lm-Y}, then
\begin{align}
\label{J1+}
\sum_{n=0}^{N-1}|J_1^n|  \le C (1+|X_0|^{(2q-1)q}_{L_\omega^{(4q-2)q}})\tau.
\end{align} 
\end{lem}

\begin{proof}
A further decomposition is introduced for $J_1^n$:
\begin{align*}
J_1^n&=\ee \int_{t_n}^{t_{n+1}} Du(T-s,\yy)(b(Y_n)-b(\yy_n)) \dd s \\
&\quad+\ee \int_{t_n}^{t_{n+1}} Du(T-s,\yy_n)(b(\yy_n)-b(\yy)) \dd s \\
&\quad+\ee \int_{t_n}^{t_{n+1}} (Du(T-s,\yy)-Du(T-s,\yy_n))(b(\yy_n)-b(\yy)) \dd s=:J_{11}^n+J_{12}^n+J_{13}^n.
\end{align*}
For the term $J_{11}^n$, by  Corollary \ref{cor-Du}, \eqref{D3-v}, Lemma \ref{lm-yy}, and  H\"older inequality, we have
\begin{align}
|J_{11}^n|
 &\le  \int_{t_n}^{t_{n+1}} e^{C (T-s)}|b(Y_n)-b(\yy_n)|_{L_\omega^2} \dd s \nonumber \\
&\le C \Big( \int_{t_n}^{t_{n+1}} e^{C (T-s)} \dd s\Big)
(1+\sup_{0\le n \le N}|Y_n|^{q^2}_{L_\omega^{2q^2}})\tau. \label{J11}
\end{align}
For $J_{12}^n$, the Taylor expansion and a conditional expectation argument give
\begin{align*} 
-J_{12}^n=\ee  \int_{t_n}^{t_{n+1}} \<Du(T-s,\yy_n),Db(\yy_n)b(Y_n)(s-t_n)+\RR_b(\yy,\yy_n)\> \dd s.
\end{align*} 
Utilizing \eqref{D3-v}, we get 
\begin{align*}
|Db(\yy_n)b(Y_n)|_{L_\omega^2} \le C|(1+|\yy_n|)^{q-1}(1+|Y_n|)^q|_{L_\omega^2}.
\end{align*}
Then  using the H\"older inequality and Lemma \ref{lm-yy} shows
\begin{align*}
(s-t_n)|Db(\yy_n)b(Y_n)|_{L_\omega^2} 
&\le C\tau |1+ |\yy_n|^{q-1}|_{L_\omega^{2r_1}}|1+|Y_n|^{q}|_{L_\omega^{2r_2}} 
\le C(1+ \sup_{0\le n \le N}|Y_n|^{q^2}_{L_\omega^{2q^2}})\tau,
\end{align*}
where $r_1=q/(q-1)$ and $r_2=q$.
Similarly, we also obtain
 \begin{align*}
|\RR_b(\yy(s),\yy_n)|_{L_\omega^2} 
&\le C |\HH_{q-2}(\yy_n+r(\yy(s)-\yy_n),\yy_n)|\yy(s)-\yy_n|^2|
_{L_\omega^2}\\
& \le C |1+|\yy_n|+|\yy(s)||^{q-2}
_{L_\omega^{2q}}|\yy(s)-\yy_n|^2_{L_\omega^{2q}} \le C(1+ \sup_{0\le n \le N}|Y_n|^{q^2}_{L_\omega^{2q^2}})\tau^{\frac{1}{q}}.
\end{align*}
Thus, by Corollary \ref{cor-Du}, one sees that
\begin{align} \label{J12}
|J_{12}^n |\le C \Big(\int_{t_n}^{t_{n+1}} e^{C (T-s)} \dd s\Big)
 (1+\sup_{0\le n \le N}|Y_n|^{q^2}_{L_\omega^{2q^2}}
)\tau^{\frac{1}{q}}.
\end{align} 
For the term $J_{13}^n$, denoting $v(r):=\yy_n+r(\yy(s)-\yy_n)$, using the Taylor expansion to $u(t,\cdot)$ yields that
\begin{align} \label{J13}
J_{13}^n=\ee \int_{t_n}^{t_{n+1}} \int_0^1D^2u(T-s,v(r))(\yy-\yy_n,b(\yy_n)-b(\yy))\dd r \dd s.
\end{align}
Applying  Corollary \ref{cor-Du}, Lemma \ref{lm-yy}, \eqref{D3-v}, and  H\"older inequality, we have
\begin{align*}
|J_{13}^n|
& \le  \int_{t_n}^{t_{n+1}} e^{C (T-s)} \sup_{s\in[0,T]}|\HH_{q-2}(X_s^{v(r)})|_{L_\omega^{2\rho_1}}|\yy-\yy_n|_{L_\omega^{2\rho_2}} 
|b(\yy_n)-b(\yy)|_{L_\omega^{2\rho_3}} \dd s \\
& \le C \Big(\int_{t_n}^{t_{n+1}} e^{C (T-s)} \dd s \Big)
(1+\sup_{0\le n \le N}|Y_n|^{(2q-1)q}_{L_\omega^{(4q-2)q}})\tau^{\frac{1}{2q-1}},
\end{align*}
 where the parameters are chosen as $\rho_1=(2q-1)/(q-2)$, $\rho_2=2q-1$, and $\rho_3=(2q-1)/q$. Then we conclude \eqref{J1} by Lemma \ref{lm-Y} and the estimates \eqref{J11}, \eqref{J12}, and \eqref{J13}.

In the case $\gamma=0$, by Corollary \ref{cor-Du}, \eqref{D3-v},   Lemma \ref{lm-yy}, and  H\"older inequality, we have
\begin{align*}
|J_{11}^n| &\le   C\int_{t_n}^{t_{n+1}} e^{-\alpha (T-s)}|b(Y_n)-b(\yy_n)|_{L_\omega^2} \dd s 
\le C\Big(\int_{t_n}^{t_{n+1}} e^{-\alpha (T-s)} \dd s \Big)
(1+\sup_{n \ge 0}|Y_n|^{q^2}_{L_\omega^{2q^2}}
)\tau.
\end{align*}
For $J_{12}^n$, combining \eqref{J12}, Corollary \ref{cor-Du},  \eqref{D3-v}, 
  H\"older inequality,  and  Lemma \ref{lm-yy}, one sees that
\begin{align*}
|J_{12}^n |\le C\Big(\int_{t_n}^{t_{n+1}} e^{-\alpha (T-s)}\dd s\Big)
 (1+\sup_{n\ge0}|Y_n|^{q^2}_{L_\omega^{2q^2}}
)\tau.
\end{align*} 
For the term $J_{13}^n$, by \eqref{J13}, Corollary \ref{cor-Du}, \eqref{D3-v},  H\"older inequality,  and Lemma \ref{lm-yy},  we choose $\rho_1=(2q-1)/(q-2)$, $\rho_2=2q-1$ and $\rho_3=(2q-1)/q$ to get 
\begin{align*}
|J_{13}^n| & \le C \Big(\int_{t_n}^{t_{n+1}} e^{-\alpha (T-s)}\dd s\Big)
 (1+\sup_{n\ge 0}|Y_n|^{(2q-1)q}_{L_\omega^{(4q-2)q}})\tau.
\end{align*} 
Then we conclude \eqref{J1+} by Lemma \ref{lm-Y} and the above three estimates. 
\end{proof}

\vspace{0.5em}

\begin{lem}
\label{lm-J2}
Assume $X_0\in L^{2q(2q-2)+2}(\Omega)$ and let Assumptions \ref{A3} and \ref{A4} hold. 
There exists a constant $C > 0$ such that
\begin{align}
\label{J2}
\sum_{n=0}^{N-1}|J_2^n| \le  e^{CT} (1+|X_0|^{q(2q-2)+1}_{L_\omega^{2q(2q-2)+2}})\tau^{\frac{1}{q^2-q+1}}.
\end{align}
Moreover, if $\gamma=0$ and Assumption \ref{A3*} holds with $2L_4^*>(2p-1)L_5$ and  $L_8>L_6\cdot C_p$ for some $p \in \nn_+$, where $C_p$ is the constant in Lemma \ref{lm-Y}, then 
\begin{align}
\label{J2+}
\sum_{n=0}^{N-1}|J_2^n| \le C (1+|X_0|^{q(2q-2)+1}_{L_\omega^{2q(2q-2)+2}})\tau.
\end{align} 
\end{lem}

\begin{proof}
For the term $J_2^n$, we here applying the following equality, for any matrix $M\in\rr^{d\times d}$ and any $a,b\in \rr^d$,
\begin{align*}
a^{\tr}Ma-b^{\tr}Mb=(a-b)^{\tr}M(a-b)+(a-b)^{\tr}Mb+b^{\tr}M(a-b).
\end{align*}
Moreover, one gives a further decomposition of $J_2^n$ as follows,
\begin{align*}
-J_2^n&=
\frac{1}{2}\sum_{j=1}^m \ee \int_{t_n}^{t_{n+1}} D^2u(T-s,\yy)
(\sigma_j(\yy)-\sigma_j(Y_n),\sigma_j(\yy)-\sigma_j(Y_n))\dd s\\
&\quad+ \frac{1}{2}\sum_{j=1}^m \ee \int_{t_n}^{t_{n+1}} D^2u(T-s,\yy)
(\sigma_j(\yy)-\sigma_j(Y_n),\sigma_j(Y_n))\dd s\\ 
&\quad+ \frac{1}{2}\sum_{j=1}^m \ee \int_{t_n}^{t_{n+1}} D^2u(T-s,\yy)
(\sigma_j(Y_n),\sigma_j(\yy)-\sigma_j(Y_n))\dd s=:J_{21}^n+J_{22}^n+J_{23}^n.
\end{align*}
For $J_{21}^n$, we apply Corollary \ref{cor-Du} with $\rho_2=\rho_3$ to obtain
\begin{align*}
|J_{21}^n |&\le  \int_{t_n}^{t_{n+1}} e^{C (T-s)} \sup_{s\in[0,T]}|\HH_{q-2}(X_s^{\yy})|_{L_\omega^{2\rho_1}} \|\sigma(\yy)-\sigma(Y_n)\|^2_{L_\omega^{2\rho_2}} \dd s,
\end{align*}
then using \eqref{sg-lip} and  H\"older inequality, for some $\rho_2 \ge 1$,
we have
\begin{align*}
\|\sigma(\yy(s))-\sigma(Y_n)\|_{L_\omega^{2\rho_2}} 
&\le \|\sigma(\yy(s))-\sigma(\yy_n)\|_{L_\omega^{2\rho_2}} +\|\sigma(\yy_n)-\sigma(Y_n)\|_{L_\omega^{2\rho_2}} \\
&\le C (1+\sup_{0\le n \le N}|Y_n|^{q}_{L_\omega^{2q\rho_2}}) \tau^{\frac{1}{2\rho_2}}.
\end{align*}
Setting $\rho_1=q/(q-2)$ and $\rho_2=q$ gives
\begin{align} \label{j21}
|J_{21}^n| \le C \Big(\int_{t_n}^{t_{n+1}} e^{C (T-s)}\dd s \Big)(1+\sup_{0\le n \le N}|Y_n|^{q^2}_{L_\omega^{2q^2}})\tau^{\frac{1}{q}}.
\end{align} 
The terms $J_{22}^n$ and $J_{23}^n$ are estimated in the same way. Therefore, taking $J_{22}^n$ as an example, using the Taylor expansion with a conditional expectation argument gives that
\begin{align*} 
J_{22}^n&= \frac{1}{2}\sum_{j=1}^m \ee \int_{t_n}^{t_{n+1}} D^2u(T-s,\yy)
(\sigma_j(\yy)-\sigma_j(\yy_n),\sigma_j(Y_n))\dd s\nonumber\\ 
&\quad+ \frac{1}{2}\sum_{j=1}^m \ee \int_{t_n}^{t_{n+1}} D^2u(T-s,\yy)
(\sigma_j(\yy_n)-\sigma_j(Y_n),\sigma_j(Y_n))\dd s\nonumber\\ 
&= \frac{1}{2}\sum_{j=1}^m \ee \int_{t_n}^{t_{n+1}} (s-t_n)D^2u(T-s,\yy_n)
(D\sigma_j(\yy_n)b(Y_n),\sigma_j(Y_n))\dd s \nonumber\\ 
&\quad+ \frac{1}{2}\sum_{j=1}^m \ee \int_{t_n}^{t_{n+1}} D^2u(T-s,\yy_n)
(\RR_{\sigma_j}(\yy,\yy_n),\sigma_j(Y_n))\dd s \nonumber\\ 
&\quad+ \frac{1}{2}\sum_{j=1}^m \ee \int_{t_n}^{t_{n+1}} (D^2u(T-s,\yy)-D^2u(T-s,\yy_n)) 
(\sigma_j(\yy)-\sigma_j(\yy_n),\sigma_j(Y_n))\dd s  \nonumber\\ 
&\quad+ \frac{1}{2}\sum_{j=1}^m \ee \int_{t_n}^{t_{n+1}} D^2u(T-s,\yy)
(\sigma_j(\yy_n)-\sigma_j(Y_n),\sigma_j(Y_n))\dd s \nonumber\\ 
&=:J_{221}^n+J_{222}^n+J_{223}^n+J_{224}^n.
\end{align*}
Applying Corollary \ref{cor-Du}, we obtain
\begin{align*}
|J_{221}^n| &\le  \tau \Big(\int_{t_n}^{t_{n+1}} e^{C (T-s)} \dd s\Big) \sup_{s\in[0,T]}|\HH_{q-2}(X_s^{\yy_n})|_{L_\omega^{2\rho_1}}  \|D\sigma(\yy_n)b(Y_n)\|_{L_\omega^{2\rho_2}}
 \|\sigma(Y_n)\|_{L_\omega^{2\rho_3}}.
\end{align*}
Choosing $\rho_1=(q^2-q+1)/(q^2-2q), \rho_2=(q^2-q+1)/q$ and $ \rho_3=q^2-q+1$, yields that
\begin{align*}
|J_{221}^n| &\le C \Big(\int_{t_n}^{t_{n+1}} e^{C (T-s)} \dd s\Big) (1+\sup_{0\le n \le N}|Y_n|^{q^2-q+1}_{L_\omega^{2q^2-2q+2}} )\tau.
\end{align*}
For $J_{222}^n$, one have
\begin{align*}
|J_{222}^n |&\le   \int_{t_n}^{t_{n+1}} e^{C (T-s)}\sup_{s\in[0,T]}|\HH_{q-2}(X_s^{\yy_n})|_{L_\omega^{2\rho_1}}
\sum_{j=1}^m|\RR_{\sigma_j}(\yy,\yy_n)|_{L_\omega^{2\rho_2}}|\sigma_j(Y_n)|_{L_\omega^{2\rho_3}} \dd s,
\end{align*}
where one gets  easily from \eqref{D3-v} and Lemma \ref{lm-yy}, for some $\rho_2 \ge 1, j\in\{1,\cdots,m\}$,
\begin{align*}
|\RR_{\sigma_j}(\yy(s),\yy_n)|_{L_\omega^{2\rho_2}} 
&\le \int_0^1 | (D\sigma_j(\yy_n+r(\yy(s)-\yy_n))-D\sigma_j(\yy_n))(\yy(s)-\yy_n)|_{L_\omega^{2\rho_2}} \dd r \\
&\le C  | (1+|\yy_n|+|\yy(s)|)^{q-2}|\yy(s)-\yy_n|^2|_{L_\omega^{2\rho_2}}  \\
&\le C (1+\sup_{0\le n \le N}|Y_n|^{q^2}_{L_\omega^{2\rho_2q}})\tau^{\frac{1}{\rho_2q}}.
\end{align*}
Thus, we have
\begin{align*}
|J_{222}^n| \le C \Big(\int_{t_n}^{t_{n+1}} e^{C (T-s)} \dd s\Big) (1+\sup_{0\le n \le N}|Y_n|^{q(2q-2)+1}_{L_\omega^{2q(2q-2)+2}}) \tau^{\frac{1}{q^2-q+1}}.
\end{align*}
Then, by the Taylor expansion and Corollary \ref{cor-Du}, we have
\begin{align*}
|J_{223}^n| 
&\le \frac{1}{2}\sum_{j=1}^m \ee \int_{t_n}^{t_{n+1}}\int_0^1 |D^3u(T-s,v_1(r)) (\yy-\yy_n,\sigma_j(\yy)-\sigma_j(\yy_n),\sigma_j(Y_n))|\dd r\dd s \\
&\le  \int_{t_n}^{t_{n+1}} e^{C (T-s)} \sup_{s\in[0,T]}|\HH_{q^*}(X_s^{v_1(r)})|_{L_\omega^{\max\{2\rho_1,\rho_3\rho_4\}}}|\yy-\yy_n|_{L_\omega^{2\rho_2}} \\
&\qquad \times \|\sigma(\yy)-\sigma(\yy_n)\|_{L_\omega^{2\rho_3\rho_5}}\|\sigma(Y_n)\|_{L_\omega^{2\rho_3\rho_6}}\dd s,
\end{align*}
where $v_1(r):=\yy_n+r(\yy(s)-\yy_n)$. Then combining with the condition \eqref{sg-lip} and Lemma \ref{lm-yy}, taking $\rho_1=(q^2+2)/(q^2-2q),\rho_2=(q^2+2)/q,\rho_3=(q^2+2)/(q+2),\rho_4=(q+2)/q$ and $\rho_5=\rho_6=q+2$, we have
\begin{align*}
|J_{223}^n |& \le C \Big(\int_{t_n}^{t_{n+1}} e^{C (T-s)} \dd s\Big) 
(1+\sup_{0\le n \le N}|Y_n|^{q^2+1}_{L_\omega^{2(q^2+2)}}) \tau^{\frac{q+1}{2q^2+4}}.
\end{align*}
By Corollary \ref{cor-Du} and the condition \eqref{sg-lip}, $J_{224}^n$ can be estimated as 
\begin{align*}
|J_{224}^n| & \le \Big(\int_{t_n}^{t_{n+1}} e^{C (T-s)} \dd s\Big) \sup_{s\in[0,T]}|\HH_{q-2}(X_s^{\yy})|_{L_\omega^{2\rho_1}}  \|\sigma(\yy_n)-\sigma(Y_n)\|_{L_\omega^{2\rho_2}}  \|\sigma(Y_n)\|_{L_\omega^{2\rho_3}}  \\
& \le C \Big(\int_{t_n}^{t_{n+1}} e^{C (T-s)} \dd s\Big) (1+\sup_{0\le n \le N}|Y_n|^{q^2-q+1}_{L_\omega^{(2q^2-2q+2)}})\tau,
\end{align*}
where $\rho_1=(q^2-q+1)/(q^2-2q),\rho_2=(q^2-q+1)/q$ and $\rho_3=q^2-q+1$. Hence, by these estimates of $|J_{221}^n|$ to $|J_{224}^n|$, we get
\begin{align} \label{j22-23}
|J_{22}^n|+|J_{23}^n| \le C \Big( \int_{t_n}^{t_{n+1}} e^{C (T-s)} \dd s\Big) (1+\sup_{0\le n \le N}|Y_n|^{q(2q-2)+1}_{L_\omega^{2q(2q-2)+2}}) \tau^{\frac{1}{q^2-q+1}}.
\end{align}
Then we conclude \eqref{J2} by Lemma \ref{lm-Y} and the estimates \eqref{j21} and \eqref{j22-23}.

In the case $\gamma=0$, for $J_{21}^n$, we apply Corollary \ref{cor-Du} with $\rho_2=\rho_3$ to obtain
\begin{align} \label{j21+}
|J_{21}^n |&\le C \int_{t_n}^{t_{n+1}} e^{-\alpha (T-s)} \sup_{s\ge 0}|\HH_{q-2}(X_s^{\yy})|_{L_\omega^{2\rho_1}}
\|\sigma(\yy)-\sigma(Y_n)\|^2_{L_\omega^{2\rho_2}}\dd s \nonumber  \\
& \le C \Big(\int_{t_n}^{t_{n+1}} e^{-\alpha (T-s)}\dd s\Big) (1+\sup_{n \ge 0}|Y_n|^{q^2}_{L_\omega^{2q^2}} )\tau.
\end{align} 
For $J_{22}^n$, applying Corollary \ref{cor-Du}, we obtain
\begin{align*}
|J_{221}^n| &\le  C \tau \Big(\int_{t_n}^{t_{n+1}} e^{-\alpha (T-s)}\dd s\Big) \sup_{s\ge 0}|\HH_{q-2}(X_s^{\yy_n})|_{L_\omega^{2\rho_1}} \|D\sigma(\yy_n)b(Y_n)\|_{L_\omega^{2\rho_2}} \|\sigma(Y_n)\|_{L_\omega^{2\rho_3}}  \\
&\le C \Big(\int_{t_n}^{t_{n+1}} e^{-\alpha (T-s)}\dd s\Big) (1+\sup_{n\ge0}|Y_n|^{q^2-q+1}_{L_\omega^{2q^2-2q+2}} )\tau,
\end{align*}
where $\rho_1=(q^2-q+1)/(q^2-2q), \rho_2=(q^2-q+1)/q$ and $ \rho_3=q^2-q+1$.
For $J_{222}^n$, one have
\begin{align*}
|J_{222}^n |&\le  \int_{t_n}^{t_{n+1}} e^{-\alpha (T-s)} \sup_{s\ge 0}|\HH_{q-2}(X_s^{\yy_n})|_{L_\omega^{2\rho_1}}
\sum_{j=1}^m|\RR_{\sigma_j}(\yy,\yy_n)|_{L_\omega^{2\rho_2}}|\sigma_j(Y_n)|_{L_\omega^{2\rho_3}}  \dd s\\
& \le C \Big(\int_{t_n}^{t_{n+1}} e^{-\alpha (T-s)}\dd s\Big) (1+\sup_{n\ge0}|Y_n|^{q(2q-2)+1}_{L_\omega^{2q(2q-2)+2}} )\tau.
\end{align*}
Then, by the Taylor expansion and Corollary \ref{cor-Du}, we have
\begin{align*}
|J_{223}^n| &\le C  \int_{t_n}^{t_{n+1}} e^{-\alpha (T-s)} \sup_{s\ge0}|\HH_{q^*}(X_s^{v_1(r)})|_{L_\omega^{\max\{2\rho_1 ,2\rho_3 \rho_4\}}} \\
&\qquad  \times |\yy-\yy_n|_{L_\omega^{2\rho_2}}\|\sigma(\yy)-\sigma(\yy_n)\|_{L_\omega^{2\rho_3\rho_5}}\|\sigma(Y_n)\|_{L_\omega^{2\rho_3\rho_6}} \dd s\\
& \le C \Big(\int_{t_n}^{t_{n+1}} e^{-\alpha (T-s)}\dd s\Big) (1+\sup_{n\ge0}|Y_n|^{q^2+1}_{L_\omega^{2(q^2+2)}} )\tau,
\end{align*}
where we set $v_1(r):=\yy_n+r(\yy(s)-\yy_n)$, $\rho_1=(q^2+2)/(q^2-2q),\rho_2=(q^2+2)/q$,$\rho_3=(q^2+2)/(q+2)$,$\rho_4=(q+2)/q$ and $\rho_5=\rho_6=q+2$. By Corollary \ref{cor-Du} and the condition \eqref{sg-lip}, $J_{224}^n$ can be estimated as 
\begin{align*}
|J_{224}^n| & \le  C \Big(\int_{t_n}^{t_{n+1}} e^{-\alpha (T-s)}\dd s\Big) \sup_{s\ge0}|\HH_{q-2}(X_s^{\yy})|_{L_\omega^{2\rho_1}}  \|\sigma(\yy_n)-\sigma(Y_n)\|_{L_\omega^{2\rho_2}}  \|\sigma(Y_n)\|_{L_\omega^{2\rho_3}} \\
& \le C \Big(\int_{t_n}^{t_{n+1}} e^{-\alpha (T-s)}\dd s\Big) 
(1+\sup_{n\ge0}|Y_n|^{q^2-q+1}_{L_\omega^{(2q^2-2q+2)}})\tau,
\end{align*}
where $\rho_1=(q^2-q+1)/(q^2-2q),\rho_2=(q^2-q+1)/q$ and $\rho_3=q^2-q+1$. Hence, by these estimates of $|J_{221}^n|$ to $|J_{224}^n|$, we get
\begin{align} \label{j22-23+}
|J_{22}^n| + |J_{23}^n| \le C \Big(\int_{t_n}^{t_{n+1}} e^{-\alpha (T-s)}\dd s\Big) 
(1+\sup_{n\ge0}|Y_n|^{q(2q-2)+1}_{L_\omega^{2q(2q-2)+2}}) \tau.
\end{align}
 Then we conclude \eqref{J2+} by Lemma \ref{lm-Y} and the estimates of \eqref{j21+} and \eqref{j22-23+}.
\end{proof}

\vspace{0.5em}

Next, we estimate the last two terms, which vanish in the jump-free case.

\vspace{0.5em}

\begin{lem}
\label{lm-J3}
Assume $X_0\in L^{2q^2}(\Omega)$ and let Assumptions \ref{A3} and \ref{A4} hold.
 Then there exists a positive constant $C$ such that
\begin{align}
\label{J3}
\sum_{n=0}^{N-1}|J_3^n| \le e^{CT}(1+|X_0|^{q^2}_{L_\omega^{2q^2}})\tau^{\frac{1}{q}}.
\end{align}
\end{lem}
\begin{proof}
For the term $J_3^n$,  we can decompose it as
\begin{align*}
-J_3^n&=\ee \int_{t_n}^{t_{n+1}} \int_Z u(T-s,\yy+\gamma(\yy_n,z))-u(T-s,\yy+\gamma(Y_n,z))\nu(\dd z)\dd s \\
&\quad+ \ee \int_{t_n}^{t_{n+1}} \int_Z u(T-s,\yy+\gamma(\yy,z))-u(T-s,\yy+\gamma(\yy_n,z))\nu(\dd z)\dd s =:J_{31}^n+J_{32}^n.
\end{align*}
For the term $J_{31}^n$, by Corollary \ref{cor-Du} and the condition \eqref{sg-lip}, we have
\begin{align} \label{j31}
|J_{31}^n| &\le  \int_{t_n}^{t_{n+1}} \int_Z e^{C (T-s)}|\gamma(\yy_n,z)-\gamma(Y_n,z)|_{L_\omega^2} \nu(\dd z)\dd s \nonumber \\
&\le C \Big(\int_{t_n}^{t_{n+1}} e^{C (T-s)} \dd s\Big) 
(1+\sup_{0\le n \le N}|Y_n|^{q}_{L_\omega^{2q}} ) \tau.
\end{align}
 Denote $v_2(r):=\yy(s)+\gamma(\yy_n,z)+r(\gamma(\yy(s),z)-\gamma(\yy_n,z))$. Using the Taylor expansion for $u(t,\cdot)$, yields 
 \begin{align*}
  J_{32}^n 
 &=\ee \int_{t_n}^{t_{n+1}} \int_Z \int_{0}^1Du(T-s,v_2(r))(\gamma(\yy,z)-\gamma(\yy_n,z))\dd r\nu(\dd z)\dd s \\
 &=\ee \int_{t_n}^{t_{n+1}} \int_Z \int_{0}^1(Du(T-s,v_2(r))-Du(T-s,v_2(r)+\yy_n-\yy))\\
 &\qquad
 (\gamma(\yy,z)-\gamma(\yy_n,z))
 \dd r\nu(\dd z)\dd s \\
 &\quad+\ee \int_{t_n}^{t_{n+1}} \int_Z \int_{0}^1(Du(T-s,v_2(r)+\yy_n-\yy)-Du(T-s,\yy_n+\gamma(\yy_n,z)))\\
 &\qquad
 (\gamma(\yy,z)-\gamma(\yy_n,z))
 \dd r\nu(\dd z)\dd s \\
 &\quad +\ee \int_{t_n}^{t_{n+1}} \int_Z Du(T-s,\yy_n+\gamma(\yy_n,z))
 (\gamma(\yy,z)-\gamma(\yy_n,z))
\nu(\dd z)\dd s \\
 &=:J_{321}^n+J_{322}^n+J_{323}^n.
 \end{align*}
 For the term $J_{321}^n$, setting $v_3(\widetilde{r})=v_2(r)+\yy_n-\yy(s)+\widetilde{r}(\yy_n-\yy(s))$, using the Taylor expansion and Corollary \ref{cor-Du}, we obtain
 \begin{align*}
 | J_{321} ^n| 
&\le \ee \int_{t_n}^{t_{n+1}} \int_Z \int_0^1 \int_0^1|D^2u(T-s,v_3(\widetilde{r}))(\yy-\yy_n, 
 \gamma(\yy,z)-\gamma(\yy_n,z))| \dd \widetilde{r}  \dd r\nu(\dd z)\dd s \\
 &\le \int_{t_n}^{t_{n+1}} e^{C (T-s)} \sup_{\widetilde{r} \in [0, 1]} \sup_{s\in[0,T]}|\HH_{q-2}(X_s^{v_3(\widetilde{r})})|_{L_\omega^{2\rho_1}}|\yy-\yy_n|_{L_\omega^{2\rho_2}}  \\
 &\qquad \times \int_Z  |\gamma(\yy,z)-\gamma(\yy_n,z)|_{L_\omega^{2\rho_3}} \nu(\dd z)\dd s \\
 &\le C \Big(\int_{t_n}^{t_{n+1}} e^{C (T-s)} \dd s\Big) 
 (1+\sup_{0\le n \le N}|Y_n|^{q^2}_{L_\omega^{2q^2}} ) \tau^{\frac{1}{q}},
 \end{align*}
 where $\rho_1=q/(q-2)$ and$\rho_2=\rho_3=q$.
 Similarly, for the estimate of $J_{322}^n$, denoting $v_4(\widetilde{r})=\yy_n+\gamma(\yy_n,z)+\widetilde{r}r(\gamma(\yy(s),z)-\gamma(\yy_n,z))$,  and combining Corollary \ref{cor-Du} with $\rho_2=\rho_3$, we have
 \begin{align*}
|  J_{322}^n |&\le \ee \int_{t_n}^{t_{n+1}} \int_Z \int_0^1 \int_0^1|D^2u(T-s,v_4(\widetilde{r}))(r(\gamma(\yy,z)-\gamma(\yy_n,z)), \\
&\qquad \qquad \qquad \qquad \qquad  \gamma(\yy,z)-\gamma(\yy_n,z))| \dd \widetilde{r}  \dd r\nu(\dd z)\dd s \\
&\le \int_{t_n}^{t_{n+1}}\int_Z e^{C (T-s)} \sup_{\widetilde{r} \in [0, 1]} \sup_{s\in[0,T]} |\HH_{q-2}(X_s^{v_4(\widetilde{r})})|_{L_\omega^{2\rho_1}} |\gamma(\yy,z)-\gamma(\yy_n,z)|^2_{L_\omega^{2\rho_2}} \nu(\dd z)\dd s \\
 &\le C \Big( \int_{t_n}^{t_{n+1}} e^{C (T-s)} \dd s \Big)(1+\sup_{0\le n \le N}|Y_n|^{q^2}_{L_\omega^{2q^2}})\tau^{\frac{1}{q}},
 \end{align*}
 where choosing $\rho_1=q/(q-2)$ and $\rho_2=\rho_3=q$.
 For $J_{323}^n$, by the Taylor expansion and taking the conditional expectation,  one gets
 \begin{align*}
 J_{323}^n&= \ee  \int_{t_n}^{t_{n+1}}\int_Z \<Du(T-s,\yy_n+\gamma(\yy_n,z)),D\gamma(\yy_n,z)b(Y_n)(s-t_n)\> \nu(\dd z)\dd s  \\
&\quad+ \ee  \int_{t_n}^{t_{n+1}}\int_Z \<Du(T-s,\yy_n+\gamma(\yy_n,z)),  \RR_{\gamma}(\yy,\yy_n)\> \nu(\dd z)\dd s,
\end{align*}
where $\RR_{\gamma}(\yy(s),\yy_n):=\int_0^1 (D\gamma(\yy_n+r(\yy(s)-\yy_n),z)-D\gamma(\yy_n,z))(\yy(s)-\yy_n)\dd r.$
By the condition \eqref{D3-v},  and Corollary  \ref{cor-Du}, we have
\begin{align*}
 |J_{323}^n| &\le  \int_{t_n}^{t_{n+1}}\int_Z e^{C (T-s)} (s-t_n)|D\gamma(\yy_n,z)b(Y_n)|_{L_\omega^2} \nu(\dd z)\dd s \\
 &\quad +\int_{t_n}^{t_{n+1}}\int_Z e^{C (T-s)}|\RR_{\gamma}(\yy,\yy_n)|_{L_\omega^2} \nu(\dd z)\dd s \\
 &\le C \Big( \int_{t_n}^{t_{n+1}} e^{C (T-s)} \dd s \Big)
 (1+\sup_{0\le n \le N}|Y_n|^{q^2}_{L_\omega^{2q^2}})\tau^{\frac{1}{q}}.
\end{align*}
 According to the estimates of  $ |J_{321}^n|-|J_{323}^n|$, it is not hard to get
 \begin{align} \label{j32}
| J_{32}^n| \le C \Big( \int_{t_n}^{t_{n+1}} e^{C (T-s)} \dd s\Big) (1+\sup_{0\le n \le N}|Y_n|^{q^2}_{L_\omega^{2q^2}})\tau^{\frac{1}{q}}.
 \end{align} 
Then we conclude \eqref{J3} by Lemma \ref{lm-Y} and the estimates \eqref{j31} and \eqref{j32} . 
\end{proof}

\vspace{0.5em}

\begin{lem}
\label{lm-J4}
Assume $X_0\in L^{2q^2}(\Omega)$ and let Assumptions  \ref{A3} and \ref{A4} hold. There exists a positive constant $C$ such that
\begin{align}
\label{J4}
\sum_{n=0}^{N-1}| J_4^n| \le  e^{C T}  (1+|X_0|^{q^2}_{L_\omega^{2q^2}})\tau^{\frac{1}{q}}.
\end{align}
\end{lem}
\begin{proof}
 For the estimate of $J_4^n$, by a further decomposition, we get
 \begin{align*}
 J_4^n&=\ee \int_{t_n}^{t_{n+1}} \int_Z Du(T-s,\yy)(\gamma(\yy_n,z)-\gamma(Y_n,z))\nu(\dd z)\dd s \\
 &\quad+\ee \int_{t_n}^{t_{n+1}} \int_Z Du(T-s,\yy_n)(\gamma(\yy,z)-\gamma(\yy_n,z))\nu(\dd z)\dd s \\
  &\quad+\ee \int_{t_n}^{t_{n+1}} \int_Z (Du(T-s,\yy)-Du(T-s,\yy_n))(\gamma(\yy,z)-\gamma(\yy_n,z))\nu(\dd z)\dd s \\
& =:J_{41}^n+J_{42}^n+J_{43}^n.
 \end{align*}
 By Corollary \ref{cor-Du} and \eqref{sg-lip}, we have
 \begin{align} \label{j41}
 |J_{41}^n|& \le C  \int_{t_n}^{t_{n+1}} e^{C (T-s)} |\gamma(\yy_n,z)-\gamma(Y_n,z) |_{L_\omega^{2}}\dd s \nonumber \\
 &\le C\Big( \int_{t_n}^{t_{n+1}} e^{C (T-s)} \dd s \Big)(1+\sup_{0\le n \le N}|Y_n|^{q}_{L_\omega^{2q}})\tau.
 \end{align}
 For $J_{42}^n$, from the Taylor expansion one obtains
 \begin{align*}
 J_{42}^n&=\ee  \int_{t_n}^{t_{n+1}}\int_Z \<Du(T-s,\yy_n), 
 D\gamma(\yy_n,z)b(Y_n)(s-t_n)\>  \nu(\dd z)\dd s\\
 &\quad +\ee  \int_{t_n}^{t_{n+1}}\int_Z \<Du(T-s,\yy_n),\RR_{\gamma}(\yy,\yy_n)\> \nu(\dd z)\dd s.
\end{align*}
By Corollary  \ref{cor-Du} and the condition \eqref{D3-v}, we have
\begin{align} \label{j42}
 |J_{42}^n| &\le  \int_{t_n}^{t_{n+1}}\int_Z e^{C (T-s)} (s-t_n)|D\gamma(\yy_n,z)b(Y_n)|_{L_\omega^2} \nu(\dd z)\dd s \nonumber \\
 &\quad +\int_{t_n}^{t_{n+1}}\int_Z e^{C (T-s)}|\RR_{\gamma}(\yy,\yy_n)|_{L_\omega^2} \nu(\dd z)\dd s \nonumber \\
 &\le C\Big(  \int_{t_n}^{t_{n+1}} e^{C (T-s)} \dd s \Big)(1+\sup_{0\le n \le N}|Y_n|^{q^2}_{L_\omega^{2q^2}})\tau^{\frac{1}{q}}.
\end{align}
For the term $J_{43}^n$, setting $v_5(r)=\yy_n+r(\yy_n-\yy(s))$, using the Taylor expansion and Corollary \ref{cor-Du}, we obtain
 \begin{align} \label{j43}
| J_{43}^n | 
 &\le \ee \int_{t_n}^{t_{n+1}} \int_Z \int_0^1|D^2u(T-s,v_5(r))(\yy-\yy_n,   \gamma(\yy,z)-\gamma(\yy_n,z)) | \dd r\nu(\dd z)\dd s \nonumber \\
 &\le \int_{t_n}^{t_{n+1}} \int_Ze^{C (T-s)} \sup_{s\in[0,T]}|\HH_{q-2}(X_s^{v_5(r)})|_{L_\omega^{2\rho_1}}|\yy-\yy_n|_{L_\omega^{2\rho_2}} \nonumber  \\
 &\qquad\qquad\qquad \times |\gamma(\yy,z)-\gamma(\yy_n,z)|_{L_\omega^{2\rho_3}} \nu(\dd z)\dd s \nonumber \\
 &\le C \Big( \int_{t_n}^{t_{n+1}} e^{C (T-s)} \dd s\Big) (1+\sup_{0\le n \le N}|Y_n|^{q^2}_{L_\omega^{2q^2}})\tau^{\frac{1}{q}},
 \end{align}
 where $\rho_1=q/(q-2)$ and$\rho_2=\rho_3=q$. 
Then we conclude \eqref{J4} by Lemma \ref{lm-Y} and the estimates \eqref{j41}, \eqref{j42}, and \eqref{j43}. 
\end{proof}

\vspace{0.5em}

With the above four estimates, we establish the following weak error estimates for the BEM \eqref{bem} relative to both the jump-diffusion and jump-free SODEs \eqref{sde}.

\vspace{0.5em}

\begin{thm}
\label{tm-weak}
Assume $X_0\in L^{2q(2q-2)+2}(\Omega)$ and $\varphi\in C_b^3(\rr^d)$ and let Assumptions \ref{A3} and \ref{A4} hold.
There exists a constant $C > 0$ such that
\begin{align} \label{weak-err}
|\ee \varphi(X_t^{\cdot})-\ee\varphi( Y_N^{X_0})| 
&\le e^{CT} (1+|X_0|^{q(2q-2)+1}_{L_\omega^{2q(2q-2)+2}})\tau^{\frac{1}{q^2-q+1}}.
\end{align}  
Moreover, if $\gamma=0$ and Assumption \ref{A3*} holds with $2L_4^*>(2p-1)L_5$ and  $L_8>L_6\cdot C_p$ for some $p \in \nn_+$, where $C_p$ is the constant in Lemma \ref{lm-Y}, then  
Then  
\begin{align} \label{weak-err+}
|\ee \varphi(X_t^{\cdot})-\ee\varphi( Y_N^{X_0})| 
&\le C(1+|X_0|^{q(2q-2)+1}_{L_\omega^{2q(2q-2)+2}})\tau.
\end{align} 
\end{thm}

\begin{proof}
Let $N \in \nn_+$.
As
$\ee \varphi(X_t^{\cdot}) =u(T, \cdot)$ and 
$\ee\varphi( Y_N^{X_0})=\ee \varphi(X_0^{Y_N^{X_0}})=u(0,Y_N^{X_0}),$
we can separate the weak error into three parts:
\begin{align*}
 |\ee\varphi(X_t^{\cdot})-\ee Y_N^{X_0}  | 
&= |\ee u(T, \cdot)-\ee u(0,Y_N^{X_0}) | \\
&\le |\ee u(0,\yy_N)  -\ee u(0,Y_N^{X_0})|
+ |\ee u(T, \cdot) -\ee u(T,\yy_0) | \\
&\quad+ | \ee u(T,\yy_0)-\ee u(0,\yy_N)| 
=:H_1+H_2+H_3.
\end{align*}
For the term $H_1$, by the construction of $\yy_N$ and the condition \eqref{D3-v}, we have
\begin{align} \label{h1}
H_1 \le \tau \ee|b(Y_N)|_{L_\omega^{2}} 
\le e^{C T}(1+|X_0|^{q}_{L_\omega^{2q}})\tau.
\end{align} 
Using the identity $\yy_0=X_0- b(X_0)\tau$, and Corollary \ref{cor-Du}, $H_2$ can be bounded as
\begin{align} \label{h2}
H_2 \le e^{C T}|\yy_0-X_0|_{L_\omega^2}
\le e^{C T}(1+|X_0|^q_{L_\omega^{2q}})\tau.
\end{align}
For the last term $H_3$, using a telescoping sum argument shows that
\begin{align*}
H_3&=\Big| \sum_{n=0}^{N-1}\ee u(T-t_{n+1},\yy_{n+1}) - \ee u(T-t_n,\yy_n) \Big|. 
\end{align*}
By the associated Kolmogorov equation \eqref{kol},  It\^o formula and \eqref{yy}, we observe that, for every $n \in\{0,1,\cdots,N-1\}$,
\begin{align*}
&u(T-t_{n+1},\yy_{n+1})-u(T-t_n,\yy_n) \\
&=- \int_{t_n}^{t_{n+1}} \partial_su(T-s,\yy) \dd s+\int_{t_n}^{t_{n+1}} Du(T-s,\yy)b(Y_n) \dd s \\
&\quad+\sum_{j=1}^m \int_{t_n}^{t_{n+1}} Du(T-s,\yy) \sigma_j(Y_n) \dd W_{j,s} + \int_{t_n}^{t_{n+1}} \int_Z Du(T-s,\yy) \gamma(Y_n,z) \widetilde{N}(\dd s, \dd z)  \\
&\quad+ \frac{1}{2}\sum_{j=1}^m  \int_{t_n}^{t_{n+1}} D^2u(T-s,\yy)
(\sigma_j(Y_n),\sigma_j(Y_n)) \dd s \\
& \quad+ \int_{t_n}^{t_{n+1}} \int_Z ( u(T-s,\yy+\gamma (Y_n,z))-u(T-s,\yy) 
-Du(T-s,\yy) \gamma(Y_n,z) )N (\dd s, \dd z) .
\end{align*}
Taking the conditional expectation argument, we have 
\begin{align*}
&\ee u(T-t_{n+1},\yy_{n+1})-\ee u(T-t_n,\yy_n) \\
&= \ee \int_{t_n}^{t_{n+1}} Du(T-s,\yy)(b(Y_n)-b(\yy)) \dd s \\
&\quad+ \frac{1}{2}\ee \int_{t_n}^{t_{n+1}} D^2u(T-s,\yy)
(\sigma_j(Y_n),\sigma_j(Y_n))
-D^2u(T-s,\yy)(\sigma_j(\yy),\sigma_j(\yy))\dd s  \\
&\quad+\ee \int_{t_n}^{t_{n+1}} \int_Z u(T-s,\yy+\gamma(Y_n,z))-u(T-s,\yy+\gamma(\yy,z))\nu(\dd z)\dd s\\
&\quad+\ee \int_{t_n}^{t_{n+1}} \int_Z Du(T-s,\yy)(\gamma(\yy,z)-\gamma(Y_n,z))\nu(\dd z)\dd s =:J_1^n+J_2^n+J_3^n+J_4^n.
\end{align*}
 According to Lemmas \ref{lm-J1}-\ref{lm-J4},  we obtain
 \begin{align} \label{h3}
H_3  & \le \sum_{n=0}^{N-1}(|J_1^n|+|J_2^n|+|J_3^n|+|J_4^n|)
\le  e^{CT} (1+|X_0|^{q(2q-2)+1}_{L_\omega^{2q(2q-2)+2}})\tau^{\frac{1}{q^2-q+1}}.
 \end{align}
Combining the estimates \eqref{h1}, \eqref{h2}, and \eqref{h3}, we conclude \eqref{weak-err}.

In the case $\gamma=0$, for the term $H_1$, we have
\begin{align} \label{h1+}
H_1 \le C(1+|X_0|^{q}_{L_\omega^{2q}})\tau.
\end{align} 
Using the identity $\yy_0=X_0- b(X_0)\tau$ and Corollary \ref{cor-Du},  
\begin{align} \label{h2+}
H_2 \le Ce^{-\alpha t}|\yy_0-X_0|_{L_\omega^2}
\le C(1+|X_0|^q_{L_\omega^{2q}})\tau.
\end{align}
For the last term $H_3$, using a telescoping sum argument, we derive 
\begin{align*}
H_3&=\Big| \sum_{n=0}^{N-1}\ee u(T-t_{n+1},\yy_{n+1})-\ee u(T-t_n,\yy_n) \Big|. 
\end{align*} 
By the associated Kolmogorov equation \eqref{kol},  It\^o formula and \eqref{yy},  we obtain
\begin{align*}
&\ee u(T-t_{n+1},\yy_{n+1})-\ee u(T-t_n,\yy_n) \\
&= \ee \int_{t_n}^{t_{n+1}} Du(T-s,\yy)(b(Y_n)-b(\yy)) \dd s \\
&\quad+ \frac{1}{2}\ee \int_{t_n}^{t_{n+1}} D^2u(T-s,\yy)
(\sigma_j(Y_n),\sigma_j(Y_n))
-D^2u(T-s,\yy)(\sigma_j(\yy),\sigma_j(\yy))\dd s.
\end{align*}
 According to Lemmas \ref{lm-J1} and \ref{lm-J2},  we arrive at
 \begin{align} \label{h3+}
H_3 \le  C (1+|X_0|^{q(2q-2)+1}_{L_\omega^{2q(2q-2)+2}})\tau.
 \end{align}
We conclude \eqref{weak-err+} by collecting the estimates \eqref{h1+}, \eqref{h2+}, and \eqref{h3+}.
\end{proof}

\vspace{0.5em}

\begin{rem}
(1) To our knowledge, higher moment bounds of the Poisson increments $\int_{t_i}^{t_{i+1}} \int_Z  \widetilde{N}(\dd s, \dd z)$ contribute at most $\mathcal O(\tau)$ \cite{CG20}, which yields $\frac{1}{2q_2}$-H\"older continuity of the jump component in the $L^{2q_2}_\omega$-sense  (see Lemma \ref{lm-yy}(1)).
    This leads to the estimates in Lemmas \ref{lm-J1}-\ref{lm-J4} and thus the dependence of the weak convergence rate on $q$ in \eqref{weak-err}.

(2) For the non-contractive case, we establish a uniform-in-time weak error estimate for the proposed schemes applied to jump-free SODEs in \cite{LWWZ25}.  
\end{rem}

\vspace{0.5em}

As a by-product of the uniform weak error estimates \eqref{weak-err+} in Theorem \ref{tm-weak}, we have the following optimal one-order convergence rate between the exact invariant measure $\pi$ of Eq. \eqref{sde} and the numerical one of the BEM \eqref{bem} in the jump-free case, which answers a question left in \cite{LL25}. 

\vspace{0.5em}

\begin{cor}\label{cor-inv}
Let Assumptions \ref{A3} and \ref{A4} hold. 
Assume that $b$ is continuous,  $2L_4^*>(2p-1)L_5$, and  $L_8>L_6\cdot C_p$ for some $p \in \nn_+$, where $C_p$ is the constant in Lemma \ref{lm-Y}. 
Then for any $\tau \in (0, 1)$ with $(2L_4^*-L_5)\tau \le 1$ and $\varphi \in C_b^3(\rr^d)$, there exists a constant $C$ such that 
\begin{align*} 
|\pi(\varphi)-\pi_\tau(\varphi)| \le C \tau.
\end{align*} 
\end{cor}

\section*{Acknowledgments}
 
The authors are supported by Guangdong Basic and Applied Basic Research Foundation, No. 2024A1515012348, and Shenzhen Basic Research Special Project (Natural Science Foundation) Basic Research (General Project), Nos. JCYJ20220530112814033 and JCYJ20240813094919026.
We also thank the referees for numerous constructive suggestions.

\bibliographystyle{amsplain}
\bibliography{bib.bib}
\end{document}